\documentclass[12pt]{amsart}

\usepackage{amsmath}
\usepackage{amsfonts}
\usepackage{amssymb}
\usepackage{latexsym}
\usepackage{graphicx}
\usepackage{enumerate}
\usepackage{multirow}
\usepackage{subfigure}

\newtheorem{theorem}{Theorem}
\newtheorem{proposition}[theorem]{Proposition}
\newtheorem{lemma}[theorem]{Lemma}
\newtheorem{cor}[theorem]{Corollary}
\newtheorem{remark}{Remark}

\usepackage{mathtools}

\makeatletter
\DeclareRobustCommand\widecheck[1]{{\mathpalette\@widecheck{#1}}}
\def\@widecheck#1#2{%
    \setbox\z@\hbox{\m@th$#1#2$}%
    \setbox\tw@\hbox{\m@th$#1%
       \widehat{%
          \vrule\@width\z@\@height\ht\z@
          \vrule\@height\z@\@width\wd\z@}$}%
    \dp\tw@-\ht\z@
    \@tempdima\ht\z@ \advance\@tempdima2\ht\tw@ \divide\@tempdima\thr@@
    \setbox\tw@\hbox{% 
       \raise\@tempdima\hbox{\scalebox{1}[-1]{\lower\@tempdima\box
\tw@}}}%
    {\ooalign{\box\tw@ \cr \box\z@}}}
\makeatother

\newcommand{\eps}{\epsilon}
\newcommand{\ep}{\epsilon}

\DeclareMathOperator{\sech}{sech}

\newcommand{\per}{{\text{per}}}

\newcommand{\be}{\begin{equation}}
\newcommand{\ee}{\end{equation}}
\newcommand{\bes}{\begin{equation*}}
\newcommand{\ees}{\end{equation*}}
\newcommand{\mand}{\quad \text{and}\quad}

\newcommand{\R}{{\mathbb{R}}}
\newcommand{\RM}{{\mathbb{R}}}
\newcommand{\C}{{\mathbb{C}}}

\newcommand{\N}{{\bf{N}}}

\renewcommand{\L}{{\mathcal{L}}}
\renewcommand{\S}{{\mathcal{S}}}
\newcommand{\A}{{\mathcal{A}}}
\renewcommand{\P}{{\mathcal{P}}}
\newcommand{\G}{{\mathcal{G}}}

\newcommand{\F}{{\mathcal{F}}}

\newcommand{\Rc}{{\mathcal{R}}}
\newcommand{\M}{{\mathcal{M}}}
\renewcommand{\O}{{\mathcal{O}}}

\newcommand{\J}{{\mathbf{J}}}

\renewcommand{\H}{{\mathcal{H}}}

\renewcommand{\tilde}{\widetilde}
\renewcommand{\hat}{\widehat}
\renewcommand{\check}{\widecheck}

\newcommand{\bunderbrace}[2]{%
  \begin{array}[t]{@{}c@{}}
  \underbrace{#1}\\
  #2
  \end{array}
}

\title{Generalized Solitary Waves in the Gravity-Capillary Whitham Equation}

\author{Mathew A. Johnson}
\address{Department of Mathematics, University of Kansas, Lawrence, KS 66049}
\email{matjohn@ku.edu}

\author{J. Douglas Wright}
\address{Department of Mathematics, Drexel University, Philadelphia, PA 19104}
\email{jdoug@math.drexel.edu}

\begin{document}

\begin{abstract}
We study the existence of traveling wave solutions to a unidirectional shallow water model which incorporates the full linear dispersion relation for both gravitational and capillary restoring forces.
Using functional analytic techniques, we show that for small surface tension (corresponding to Bond numbers between $0$ and ${1}/{3}$)
there exists small amplitude solitary waves that decay to asymptotically small periodic waves at spatial infinity.  The size of the oscillations
in the far field are shown to be small beyond all algebraic orders in the amplitude of the wave.  We also present numerical evidence, based on the recent
analytical work of Hur \& Johnson, that the asymptotic end states are modulationally stable for all Bond numbers between $0$ and $1/3$. %{\bf Not yet we do not!}
%
%We study the existence of generalized solitary waves in a fully dispersive shallow water model combining the (uni-directional) dispersion relation
%for gravity-capillary water waves with a canonical shallow water nonlinearity.  In particular, we use functional analytic techniques to prove the existence
%of a family of solitary waves which decay to asymptotically small periodic waves at spatial infinity.  
\end{abstract}

\maketitle

\section{Introduction}

\subsection{``Full dispersion" models}
It is well known that the Korteweg-de Vries (KdV) equation 
\begin{equation}\label{kdv}
u_t+\sqrt{gd}\left(1+\frac{1}{6}d^2\partial_x^2\right)u_x+uu_x,
\end{equation}
approximates the full water wave problem in the small amplitude, long wavelength regime \cite[Section 7.4.5]{LannesBook}
\cite{Schneider-Wayne00} \cite{Schneider-Wayne02} \cite{Dull}.
Here, $u(x,t)$ corresponds to the fluid height at spatial position $x$ at time $t$, $d$ corresponds
to the undisturbed depth of the fluid, and $g$ is the acceleration due to gravity. 
At least in part, the agreement in this asymptotic regime can be understood by noting that the phase speed of the water wave problem 
expands for $|kd|\ll 1$ as
\[
c_{ww}(k):=\pm\sqrt{g\frac{\tanh(kd)}{k}}=\pm\sqrt{gd}\left(1-\frac{1}{6}k^2d^2\right)+\mathcal{O}(|kd|^4),
\]
so that the KdV phase speed agrees to second order in $|kd|$ with that of the full water wave problem.  

The KdV equation admits both solitary and periodic traveling wave solutions which are 
nonlinearly stable in appropriate senses \cite{Bona75} \cite{PW94} and these results have pointed the way towards (at least some) similar results for the full water wave problem
\cite{Beale79} \cite{PS16}.
Naturally, however, the KdV phase speed is a terrible approximation of $c_{ww}$ for
even moderate frequencies.  It should come as no surprise then that  KdV fails to exhibit many high-frequency phenomena\footnote{That is, occurring for $kd$ not asymptotically small.}
such as wave breaking -- the evolutionary formation of bounded solutions with infinite gradients -- and peaking -- the existence of bounded, steady solutions with a singular point, such as a peak or a cusp.

The above observations led Whitham \cite{Whitham_book} to state {\it ``It is intriguing to know what kind of simpler mathematical equations (than the physical problem) could include
[peaking and breaking]."}
In response to his own question, Whitham put forward the model
\begin{equation}\label{whitham}
u_t+\left(\mathcal{W}_{gd}u+u^2\right)_x=0, \quad u = u(x,t) \in \R,\ x \in \R,\ t \in \R,
\end{equation}
where here $\mathcal{W}$ is a Fourier multiplier operator on $L^2(\RM)$ defined via
\[
\mathcal{F}\left(\mathcal{W}_{gd}f\right)(k)=\sqrt{\frac{g\tanh(k d)}{k}}~\hat{f}(k):
\]
see \cite[p. 477]{Whitham_book}
By construction, the above pseudodifferential equation, modernly referred to as the ``Whitham equation" or ``full dispersion equation", has a  phase speed that agrees exactly with that of the full  water wave problem.
Since \eqref{whitham} balances both the full water wave dispersion with a canonical shallow water nonlinearity, Whitham conjectured that the equation \eqref{whitham} would be capable
of predicting both breaking and peaking of waves.

And in fact it does.  The Whitham equation \eqref{whitham} has recently been shown\footnote{See also \cite{NS94} and \cite{CE98} for related results, and the discussion in \cite{Hur1}.} 
to exhibit wave breaking \cite{Hur1}, as well as to admit both periodic \cite{EK09} and solitary \cite{EGW,SW} waves.
In particular, %For instance, \eqref{whitham} admits smooth solitary and periodic traveling wave solutions.  In particular, 
in \cite{EK13,EW}, the authors conducted a detailed global bifurcation analysis of periodic traveling waves for \eqref{whitham} and concluded 
that the branch of smooth periodic waves terminates in a non-trivial cusped solution -- bounded solution with unbounded derivative\footnote{The wave behaves like $|x|^{1/2}$ 
near the cusp.}-- that is monotone and smooth on either side of the cusp.  
Additionally, its well-posedness was addressed in \cite{EEP15}, and in \cite{HJ15} it was shown that \eqref{whitham} bears 
out the famous Benjamin-Feir, or modulational, instability of small amplitude periodic
traveling waves; see also the related numerical work \cite{SKCK14} on the stability of large amplitude periodic waves.  
Taken together it is clear that, regardless of its rigorous relation to the full water wave problem\footnote{The relevance of the Whitham equation as a model
for water waves was recently studied in \cite{MDK15}, where it was found to perform better than the KdV and BBM equations in describing the surface of waves in the
intermediate and short wave regime.}, the fully dispersive model \eqref{whitham} admits many interesting
high-frequency phenomena known to exist in the full water wave problem. %, both with respect to existence of solutions as well as dynamics.
%{\bf Wave breaking?}

\subsection{Including surface tension} 
It is thus natural to consider the existence and behavior of solutions when additional physical effects are included.  In this paper, we incorporate surface tension and consider
the following pseudodifferential equation
\be\label{dimW}
u_t + (\M_{gd\tau} u + u^2)_x = 0.
\ee
Here, $u$, $x$, and $t$ are as in \eqref{whitham} above, and $\mathcal{M}_{gd\tau}$ is a Fourier multiplier operator on $L^2(\RM)$ with symbol
\[
m_{gd\tau}(k) = \sqrt{\left(g+\tau k^2\right) {\tanh(k d) \over k}}.
\]
This symbol gives exactly the phase speed for the full gravity-capillary wave problem in the irrotational setting \cite{J97,Whitham_book}.
The parameter $\tau>0$ is the coefficient of surface tension,
%\footnote{Note at the air-sea interface, $g\approx 9.81m/s^2$ and 
%$\tau\approx 7.3\times 10^{-3}N/m$.  Thus, surface tension effects are negligible for waves with wavelength
%several times greater than $2\pi\sqrt{T/g}\approx 1.7 cm$. {\bf I do not agree with this statement and its inclusion here undercuts our point: inclusion of surface tension is a *singular* perturbation and thus even at small values has nontrivial and substantive effects.}}, 
while both $g$ and $d$ are as in \eqref{whitham}.
The  properties of the symbol $m_{gd\tau}$ above depend on the non-dimensional ratio 
\[
\beta:=\tau/gd^2,
\]
which is referred to as the Bond number.  
When $\beta>1/3$, corresponding to ``strong" surface tension,
the phase speed $m_{gd\tau}(k)$ is monotone increasing for $k>0$ with high-frequency asymptotics $m_{gd\tau}(k)\approx  |\tau k|^{1/2}$ for $|k|\gg 1$, while for ``weak" surface tension,
corresponding to $0<\beta<1/3$, $m_{gd\tau}(k)$ has a unique positive global minimum, after which it is monotonically increasing with the same high-frequency behavior: see Figure \ref{fig:disp}.
Concerning its relation to the full water wave equations, see \cite{Carter17}, where there the author studies the accuracy of \eqref{dimW} in modeling real-world experiments of waves on shallow water.

\begin{figure}
\begin{center}
(a)~~\includegraphics[scale=0.5]{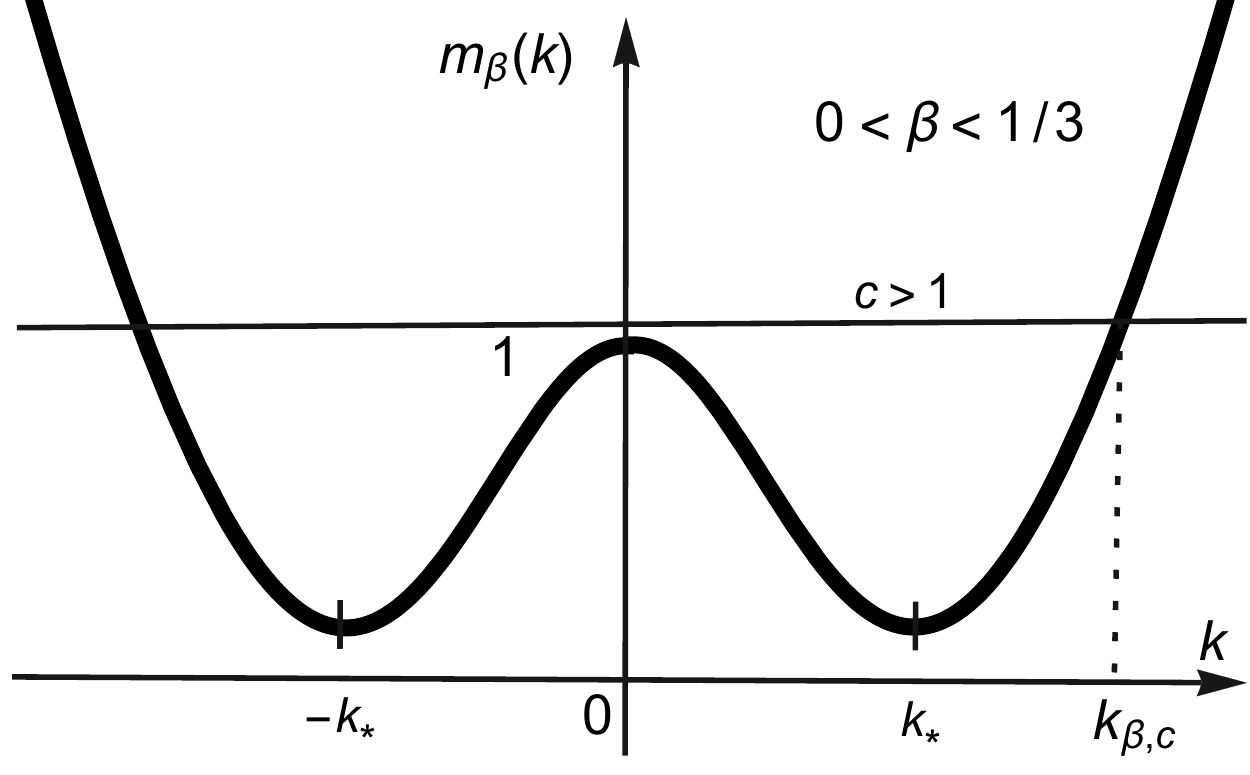}\qquad(b)~~\includegraphics[scale=0.5]{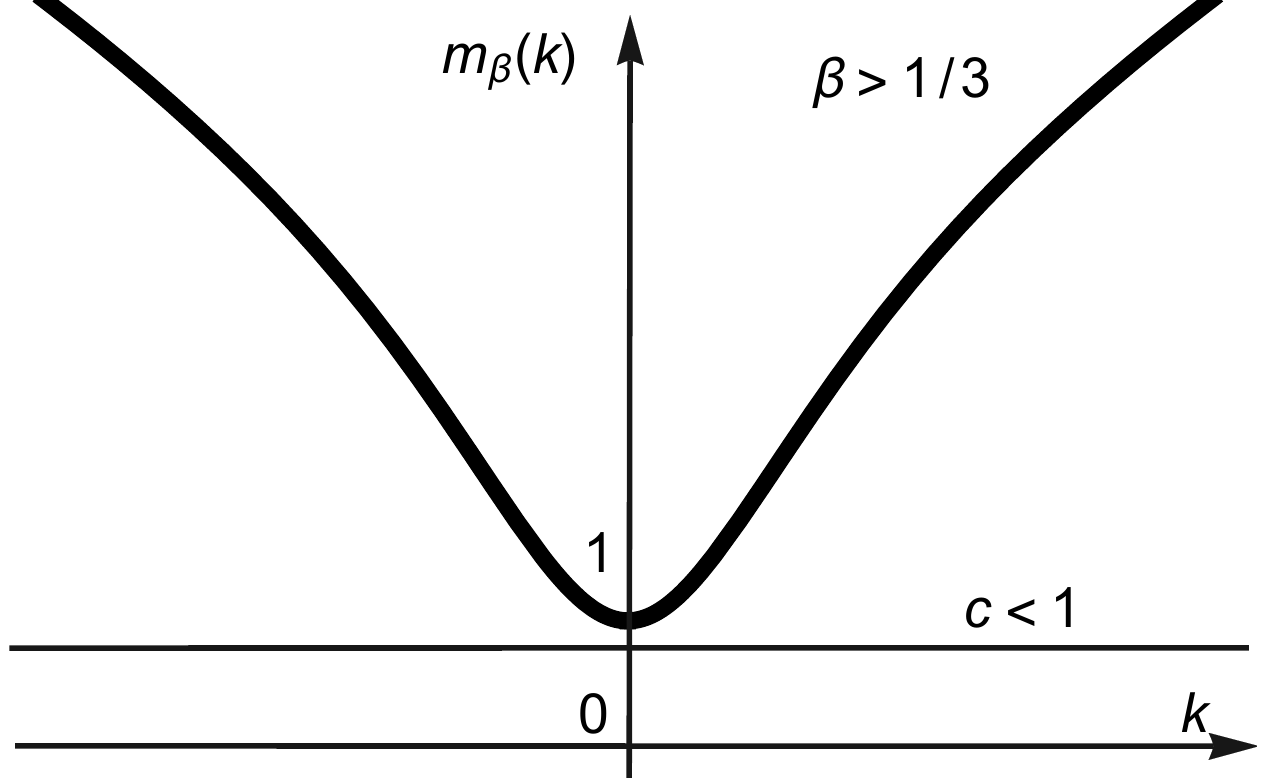}
\end{center}
\caption{Schematic drawings of the linear phase speed $m_\beta(k)$ associated to \eqref{dimW} for both (a) small surface tension, corresponding
to $\beta\in(0,1/3)$, and (b) large surface tension, corresponding to $\beta>1/3$.}\label{fig:disp}
\end{figure}

\

%In this paper, we are interested in the existence of solitary wave solutions of the gravity-capillary Whitham equation \eqref{dimW}.
In the full gravity-capillary wave problem with large surface tension ({\it i.e.} $\beta>1/3$) there exist subcritical\footnote{In this context, a traveling wave is ``subcritical"
if its speed is less than the long wave speed $c_{lw}:=m_\beta(0)=\sqrt{gd}.$ If the traveling wave's speed is greater than $c_{lw}$, then it is said to be ``supercritical."}
%That is, they should have
%have wave speeds strictly less than the long wave speed $m_\beta(0)=\sqrt{gd}$, corresponding to the global minimum of the phase speed $m_\beta(k)$.}
solitary waves of depression (that is, they are asymptotically zero but with a unique critical point corresponding to a strictly negative absolute minimum). See, for instance, \cite{AK85,AK89}.
When $\beta\in(0,1/3)$, however, considerably less is known about the existence of %asymptotically constant 
genuinely localized solitary waves.  
What is known is that supercritical  \emph{generalized solitary waves}  (also called \emph{nanopterons}) exist in this setting. 
That is, waves that are roughly the superposition of a solitary wave and a (co-propagating) periodic wave of substantially smaller amplitude, dubbed ``the ripple." See \cite{Sun91} \cite{Beal91}.\footnote{When $\beta$ is less than, but close to $1/3$, very precise rigorous asymptotics on the size of the ripple have been established, see \cite{SS93} \cite{Lombardi}.}
In this paper, we establish analogs of these results for the gravity-capillary Whitham equation \eqref{dimW}.  We also note that the existence and stability of periodic
traveling waves in \eqref{dimW} have recently been investigated in \cite{HJ15b,RKH17}.  In Section \ref{S:stability} below, we will apply these stability 
results to make observations concerning the stability of the generalized solitary waves constructed here.  

\subsection{Formal computations and the main results}

A routine nondimensionalization of \eqref{dimW} converts it to
\be\label{nondimW}
u_t + \left(\mathcal{M}_{\beta} u + u^2\right)_x = 0,
\ee
where $\mathcal{M}_\beta$ is the Fourier multiplier operator with symbol
\be
m_\beta(k) := \sqrt{\left(1+\beta k^2\right) {\tanh(k ) \over k}}.
\ee
We will henceforth be working with this version of the system. 
Substituting the traveling wave ansatz $u(x,t) = w(x-ct)$ into \eqref{dimW} yields, after one integration the nonlocal profile equation
\be\label{TWE1}
(\M_\beta-c) w + w^2 = 0.
\ee

We are interested in long wavelength/small amplitude solutions of \eqref{TWE1}. Consequently we expect the wave speed $c$ to be close
to the long wave speed $c_{lw}$, which in the nondimensionalized problem is exactly one. And so for $0<\eps\ll 1$, we make a ``long wave/small amplitude/nearly critical" scaling of \eqref{TWE1} by setting 
\be\label{lw scaling}
c = 1 + {\gamma_\beta}\epsilon^2 \mand w(y) = \ep^2 W( \ep y).
\ee
In the above we have made the (convenient) choice 
\[
\gamma_\beta:=-{1 \over 2} m_\beta''(0) = {1-3\beta \over 6}.
\]
Consequently,  if $\beta>1/3$ the solutions we are looking for are slightly subcritical, since $\gamma_\beta < 0$. If $\beta \in (0,1/3)$ then the solutions are supercritical.
%Note that it is positive so that $c$ here is  slightly bigger than $1$, which is ``the speed of sound."
After applying \eqref{lw scaling}, \eqref{TWE1} becomes
\be\label{TWE2}
{(\M^\ep_\beta-1 -\gamma_\beta \ep^2)} W + \ep^2 W^2 =0
\ee
where $\M^\ep_\beta$ is a Fourier multiplier with symbol $ m_\beta(\ep K)$.

For $0<\eps\ll 1$ we have the expansion
\be\label{expando}
m_\beta(\ep K) = 1 - \gamma_\beta \ep^2 K^2 + \O(\ep^4K^4).
\ee
With the usual Fourier correspondence of $\partial_X$ and $i K$, the above indicates the following formal expansion:
\be\label{expando2}
\M_\beta^\ep= 1 + \gamma_\beta \ep^2 \partial_X^2 + \O(\ep^4 \partial_X^4).
\ee
Therefore the (rescaled) profile equation
\eqref{TWE2} formally looks like
\begin{equation}\label{formalprofile}
\left(\gamma_\beta  W'' - \gamma_\beta  W +  \O(\ep^2)\right)  + \ep^2 W^2  =0.
\end{equation}
Putting $\ep = 0$, it follows that the solution % $W_0:=W|_{\eps=0}$ 
satisfies the ODE
\be\label{KDVTWE}
W'' - W + \gamma_\beta^{-1} W^2 = 0.
\ee
We immediately recognize \eqref{KDVTWE} as the profile equation  associated with solitary wave solutions 
of the (suitably rescaled) KdV equation \eqref{kdv}.  
In particular,  \eqref{KDVTWE} admits a unique non-trivial even solution in $L^2(\R)$ given by 
\[
\sigma_\beta(X) :=  {3 \gamma_\beta \over 2 } \sech^2\left( {X \over 2} \right) = {1-3 \beta \over 4}  \sech^2\left( {X \over 2} \right). 
\]
Note that $\sigma_\beta(X)$ is positive when $\beta\in(0,1/3)$ and negative when $\beta > 1/3$, corresponding to solitary waves % $\sigma_\beta(X)$ is negative: a solitary wave of depression.
of elevation and depression, respectively.

\

Our main goal is to analyze to how  $\sigma_\beta$ deforms for $0<\eps\ll 1$.  The main difficulty in the analysis is that the expansion \eqref{expando} is \emph{not uniform} in
$K$ and, as a consequence, that the ODE \eqref{KDVTWE} is necessarily singularly perturbed by the $\mathcal{O}(\eps^2)$ terms in \eqref{formalprofile}.  

It turns out that when $\beta > 1/3$ there is a straightforward way to ``desingularize" the problem. The main observation is that
the multiplier for the operator $\M_\beta - c$ is non-zero for all wave numbers when $c<1$ ({\it i.e.} is subcritical) when $\beta > 1/3$.
Therefore the linear part of \eqref{TWE1} can be inverted. 
Doing so and then implementing the scaling \eqref{lw scaling} results in a system which is not singularly perturbed in $\ep$
and one can use the implicit function theorem to continue the solution $\sigma_\beta$ to $\ep > 0$.
In the recent paper by Stefanov \& Wright \cite{SW} 
this strategy (which was inspired by \cite{FP99,FML15}) was deployed for a class of pseudodifferential equations which includes \eqref{TWE1} when $\beta > 1/3$.
Their main result can be directly applied here.  
We explain this in greater detail in Section \ref{depression} but, for now, here is our result:
\begin{cor}\label{depression cor} 
There exist subcritical solitary waves of depression for \eqref{nondimW} when the capillary effects are strong. Specifically, for all $\beta > 1/3$ and all $\ep$ sufficiently close to zero, there exists a small amplitude, localized, smooth, even function $R_\ep$ 
%\footnote{Theorem \ref{ASthm} only tells us that the solutions are in $E^1_0$. In this problem, however, we have the additional information that $m_\beta(k)$ grows like $|k|^{1/2}$ for large $|k|.$ With this, a straightforward bootstrapping argument demonstrates that the solutions are smooth. We omit these details.} 
%$(w,c) \in \left(\cap_{r\ge0} E^r_0\right) \times \R$  
such that 
$$
w_\ep(x) = -\left({3\beta -1\over 4}\right)\ep^2  \sech^2\left( \ep x \over 2 \right) + R_\ep(\ep x) \mand c_\ep = 1 - \left( {3\beta -1 \over 6}  \right) \ep^2
$$
solve
\eqref{TWE1}.
%we have
%$$
%W_\ep(X) = -{3\beta -1\over 4} \sech^2\left( X \over 2 \right) + R_\ep(X),
%$$
%where 
For any $r \ge 0$ there exists $C_r>0$ such that  $$\| R^{(r)}_\ep\|_{L^2(\R)} \le C_r \ep^4.$$
$R_\ep$ is the unique function with the aforementioned properties.
\end{cor}
%In the above that, since $\beta > 1/3$, the wave speed $c$ is strictly less than $1$ and the leading order contribution in $W_\ep$
%is strictly negative. That is to say, this result indicates that there is a subcritical solitary wave of depression for \eqref{dimW} when $\beta > 1/3$.

On the other hand, when $\beta \in (0,1/3)$ a similar desingularization will not work. In Figure \ref{fig:disp}, note that when $\beta \in (0,1/3)$ and $c>1$ ({\it i.e.} is supercritical)
there is a unique $k_{\beta,c}>0$ at which 
\be\label{critical freq}
m_\beta(\pm k_{\beta,c}) - c = 0.
\ee
 Thus $\M_\beta - c$ cannot be inverted; the situation becomes  more complicated. What occurs is that when $\ep > 0$ the main pulse $\sigma_\beta$, through a sort of weak resonance, excites a very small amplitude periodic wave with frequency close to $k_{\beta,c}$. The end result is a generalized solitary wave as described above. See Figure \ref{fig:schematic} for a sketch of the solution.
 Our proof  is modeled on the one devised by Beale in \cite{Beal91} to study traveling waves in the full gravity-capillary problem (and which has subsequently
 been deployed to study generalized solitary waves in other contexts in  \cite{Faver} \cite{Faver-Wright} \cite{Hoffman-Wright} \cite{Amick-Toland}). 
The proof is found in Section \ref{the general}.
 Here is our result:
 \begin{theorem}\label{main theorem}
 There exist supercritical generalized solitary waves  for \eqref{nondimW} when the capillary effects are weak.
 Specifically, for all $\beta \in (0,1/3)$ and all $\ep$ sufficiently close to zero, there exist smooth, even functions
 $R_\ep$ and $P_\ep$ such that
 \be\label{form}
w_\ep(x) = \left({1- 3\beta\over 4}\right)\ep^2  \sech^2\left( \ep x \over 2 \right) + R_\ep(\ep x) + P_\ep(x) \mand c_\ep = 1 + \left( {1-3\beta \over 6}  \right) \ep^2
\ee
solve
\eqref{TWE1}.
The functions $R_\ep$ and $P_\ep$ have the following properties.
\begin{enumerate}[(i)]
\item $R_\ep(X)$ is an exponentially localized function of small amplitude. In particular, 
there is a constant $q_* > 0$ such that for all $r\ge0$ there exists $C_r > 0$ for which
\bes%\label{core}
\| \cosh^{q_*}(\cdot) R^{(r)}_\ep(\cdot) \|_{L^\infty(\R)} \le C_r \ep^4.
\ees
\item $P_\ep(x)$ is a periodic solution of \eqref{TWE1} whose frequency is approximately $k_{\beta,c_\ep}$ and 
whose amplitude is small beyond all algebraic orders. Specifically, there is a constant $\delta>0$ such that 
the frequency of $P_\ep$ lies in the interval $[k_{\beta,c_\ep} - \delta \ep,k_{\beta,c_\ep} +\delta \ep]$ and for all $r \ge 0$ there is a constant $C_{r}>0$ for which
\bes%\label{beyond all orders}
\| P_\ep \|_{L^\infty(\R)} \le C_{r} \ep^r.
\ees
\end{enumerate}
Moreover, this solution is unique in the sense that no other pair $(R_\ep,P_\ep)$ leads to a solution of \eqref{TWE1} of the form \eqref{form}
which meets all the criteria stated in (i) and (ii).
\end{theorem}

\begin{figure}
\includegraphics[scale=0.7]{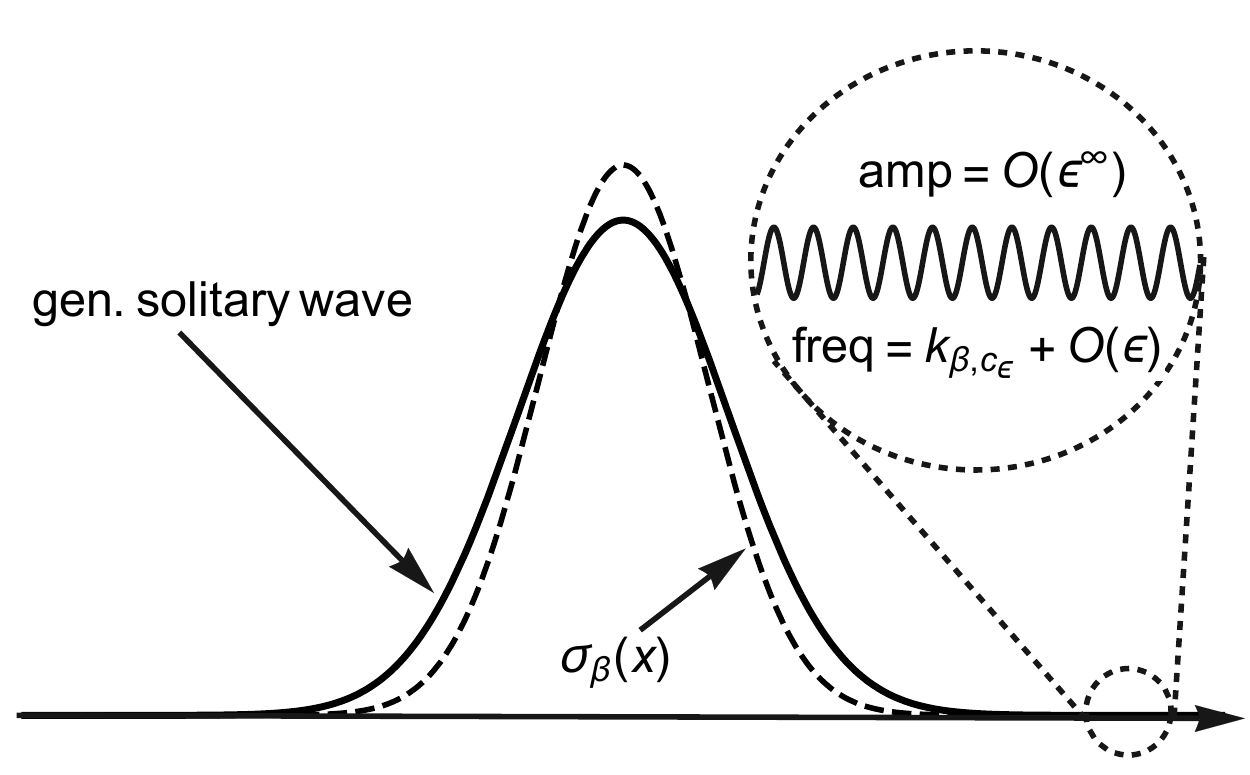}
\caption{This is a cartoon illustrating our main result Theorem \ref{main theorem}.}\label{fig:schematic}
\end{figure}

The waves $w_\eps$ constructed in Theorem \ref{main theorem} consist of a localized cure connecting two asymptotically small oscillatory end states.  In addition to the existence
result presented above, we also discuss how this rough ``decomposition" can be combined with recent work \cite{HJ15b} to provide insight into the possible stability of these waves.  While we
do not arrive at any definitive stability results here, we hope our study spurs additional work.

\begin{remark} %{\bf (I haven't altered this, just relocated it and put it in Remark form)}
We also point out the work \cite{A16}, where the author uses direct variational arguments
to prove the existence of localized solutions to a large class of pseudo-partial differential equations that contains \eqref{nondimW} for all values of $\beta>0$.  In the case of small surface tension,
however, the fact that his waves have phase speeds slightly less than the global minimia of the phase speed $m$ (hence, are necessiarly subcritical with respect to the long-wave phase
speed $m_\beta(0)$).  Labeling $k^*$ the (strictly positive) frequency where the global minima of 
$m$ is achieved, it follows that the localized waves constructed in \cite{A16} are in fact \emph{modulated solitary waves}, taking approximately the form $\phi(x)\cos(k^*x)$
for some localized function $\phi\in L^2(\RM)$. %(Does Saut have some results?).  
\end{remark}

\begin{center}
\underline{Acknowledgments}
\end{center}
%The authors would like to thank the Kansas City Symphony for its hospitality and for the beautiful melodies (except for that experimental crap {(\bf I am pro-experimental crap}\footnote{In general I am too, 
%but theirs just wasn't good... I suppose we can agree to disagree :-)}) it provided during the initial discussions that led to this work.  Seriously though,
The authors would like to thank Mats Ehrnstr\"om, Mark Groves, Miles Wheeler, Atanas Stefanov and Mathias Arnesen for useful conversations about this work.  The work of M.A.J. was partially
supported by the NSF under grant DMS-1614785.
The work of J.D.W. was partially
supported by the NSF under grant DMS-1511488.

\section{Conventions}

Here we specify the notation for the function spaces we will be using along with some other conventions.

\subsection{Periodic functions}
We let $W^{r,p}_\per:=W^{r,p}{(\mathbb{T}})$ be  the usual ``$r,p$" Sobolev
space of $2 \pi$-periodic functions. We denote $L^p_\per:=W^{0,p}_\per$ and
 $H^r_\per:=W^{r,2}_\per$.
Put
\be\label{per spaces}
E^r_\per:=H^r_{\per} \cap \left\{\text{even functions} \right\}.
\ee
%We will also make use of 
%\be\label{per0 spaces}
%E^s_{\per,0}:=E^s_{\per} \cap \left\{u(X):\int_{-\pi}^\pi u(X) dX = 0 \right\}.
%\ee
%That is to say, mean-zero even periodic functions.
By $C^r_\per$ we mean the space of $r-$times differentiable $2\pi$-periodic functions
and $C^\infty_\per$ is the space of smooth $2\pi$-periodic functions. 

\subsection{Functions on $\R$:}
We let $W^{r,p}:=W^{r,p}(\R)$ be the usual ``$r,p$" Sobolev
space of functions defined on  $\R$. For $q \in \R$ put
\be\label{Wspb spaces}
W^{r,p}_{q} := \left\{ u \in L^2(\R) : \cosh^q(x) u(x) \in W^{r,p}(\R) \right\} .
\ee
These are Banach spaces with the naturally defined norm. 
If we say a function is ``exponentially localized" we mean that it is in one of these spaces with $q > 0$.

 Put $L^p_q:=W^{0,p}_q$, 
$H^r:=W^{r,2}$ and $H^r_q:=W^{r,2}_q$ and 
denote $$\| f\|_{r,q}:=\| f\|_{H^r_q}:=\|\cosh(\cdot)^q f(\cdot)\|_{H^r(\R)}.$$ We let 
\be\label{sb spaces}
E^r_q:=H^r_{q} \cap \left\{\text{even functions} \right\}.
\ee

\subsection{Spaces of operators} For Banach spaces $X$ and $Y$ we let $B(X,Y)$ be the space
of bounded linear operators from $X$ to $Y$ equipped with the usual induced topology.

\subsection{Big $C$ notation:}
Suppose that 
$Q_1$ and $Q_2$ are positive quantities (like norms) which depend upon 
the smallness parameter $\ep$, the regularity index $r$, the decay rate $q$ and some collection of elements $\eta$ which live in a Banach space $X$.

When we write ``$Q_1 \le C Q_2$" we mean ``there exists $C>0$,
$\ep_0 > 0$, $q_0>0$, $\delta >0$ such that $Q_1 \le C Q_2$ for all $r \ge 0$, $q \in [0,q_0]$, $\ep \in (0,\ep_0]$ and $\|\eta\|_{X} \le \delta_0$."
In particular, the constant does not depend on anything.

When we write ``$Q \le C_r f(\ep)$" we mean 
``there exists $\ep_0 > 0$, $q_0 > 0$, $\delta_0>0$ such that for any $r \ge 0$ there exist $C_r>0$ such that 
$Q_1 \le C_r Q_2$ when $q \in [0,q_0]$, $\ep \in (0,\ep_0]$ and $\|\eta\|_{X} \le \delta_0$."
In this case, the constant depends on $r$ but nothing else.

When we write ``$Q \le C_q f(\ep)$" we mean 
``there exists $\ep_0 > 0$, $q_0 > 0$, $\delta_0>0$ such that for any $q \in (0,q_0]$ there exists $C_q>0$ such that 
$Q_1 \le C_q Q_2$ when $r \ge 0$, $\ep \in (0,\ep_0]$ and $\|\eta\|_{X} \le \delta_0$."
The constant depends on $q$ but nothing else.

Lastly, when we write ``$Q_1 \le C_{r,q} Q_2$" we mean 
``there exists $\ep_0 > 0$, $q_0 > 0$, $\delta_0>0$ such that for any $r \ge 0$ and $q \in (0,q_0]$ there exist $C_{r,q}>0$ such that 
$Q_1 \le C_{r,q} Q_2$ when $\ep \in (0,\ep_0]$ and $\|\eta\|_{X} \le \delta_0$."
That is, the constant depends on $r$ and $q$.

%\noindent-{\bf Big $\O$ notation:}
%For an $\ep$ dependent quantity $q_\ep$ in a Banach space $X$ 
%and a function $f: \R^+ \to \R^+$ we write
%$q_\ep \le \O_X(f(\ep))$ if there exists 
%$\ep_0>0$ and $C>0$ such that $\ep \in (0,\ep_0]$ implies $\|q_\ep\|_X \le C f(\ep)$.
%Likewise we write $q_\ep \ge \O_X(f(\ep))$ if there exists 
%$\ep_0>0$ and $C>0$ such that $\ep \in (0,\ep_0]$ implies $\|q_\ep\|_X \ge C f(\ep)$.
%We write $q_\ep = \O_X(f(\ep))$ if $q_\ep \le \O_X(f(\ep))$ and $q_\ep \ge \O_X(f(\ep))$.
%We write $q_\ep \le \O_X(\ep^\infty)$ if $q_\ep \le \O_X(\ep^n)$ for all $n \ge 0$.

\subsection{Fourier analysis:}
We use following normalizations and notations for the Fourier transform and its inverse:
$$
\hat{f}(k):=\F[f](k):={1 \over 2 \pi} \int_\R f(x)e^{-ikx} dx \mand
\check{g}(x):=\F^{-1}[g](x) := \int_\R g(k)e^{ikx} dk.
$$
%Likewise, we use the following normalizations and notations for the Fourier series of
%a $2 \pi$-periodic function:
%$$
%\hat{f}(k):={1 \over 2\pi} \int_{-\pi}^\pi f(x) e^{-ikx} dx \mand \check{g}(x) := \sum_{k \in \Z} g(k) e^{ik  x}.
%$$
%Likewise, we use the following normalizations and notations for the Fourier series of
%a $2 P$-periodic function:
%$$
%\hat{f}(k):={1 \over 2P} \int_{-P}^P f(x) e^{-ik\pi x/P} dx \mand {f}(x) = \sum_{k \in \Z} \hat{f}(k) e^{ik \pi x/P}.
%$$
%We have used the same ``hat" notation for the Fourier transform and the coefficients of the Fourier series; context
%will always make it clear which we mean. 

\section{Solitary waves of depression when $\beta > 1/3$.}\label{depression}

The following theorem on the existence of solitary waves in a certain class of pseudodifferential equations was proved in \cite{SW}: 
\begin{theorem} \label{ASthm} 
Suppose that  there exists $\delta_*>0$ such that 
$n: (-\delta_*,\delta_*) \to \R$ is $C^{2,1}$ (that is, its second derivative exists and is uniformly Lipschitz continuous) and satisfies
$$
n(0) = n'(0) = 0 \mand n''(0) > 0.
$$
Moreover
suppose that 
$l:\R \to \R$ is  even and there exists $\bar{k}>0$ 
which has the following properties:
\begin{enumerate}[(a)]
\item $l(k)$ is $C^{3,1}$ (that is, its third derivative exists and is uniformly Lipschitz continuous) for $k \in [-\bar{k},\bar{k}]$.
\item
$
l_2:=\max_{|k|\le \bar{k}} l''(k) < 0.
$
%In particular $l''(0)<0$. 
\item
$
l_1:=\sup_{k \ge \bar{k}} l(k)  < l(0).
$
\end{enumerate}

Let $\mathcal{L}$ be the Fourier multiplier operator with symbol $l(k)$.
Then there exists $\ep_0>0$, so that for every $\ep\in (0, \ep_0)$, there is a solution $(\phi,v) \in E^1_0 \times \R$ 
of 
\be\label{AS eqn}
\left(\mathcal{L} -v\right) \phi + n(\phi) = 0
\ee
of the form
$$
\phi(y)= \ep^2 \Phi_\ep(\ep y) \mand v=  l(0) - {1 \over 2} l''(0) \ep^2
$$
where
$$
\Phi_\ep(X) =  -\frac{3 l''(0)}{2 n''(0)}  \sech^2\left(\frac{X}{2}\right)+ \rho_\ep(X).
$$
The function $\rho_\ep(X) \in E^1_0$ satisfies the estimate 
$
\left\|\rho_\ep\right\|_{1,0} \le C\ep^2.
$
\end{theorem}

This theorem can be applied directly to \eqref{TWE1} when $\beta >1/3$. Specifically, let $v = -c$, $l = -m$, $\mathcal{L} = -\M$, $\phi = - u$ and $n(\phi) = \phi^2$. Then \eqref{TWE1} is transformed into \eqref{AS eqn}. Clearly $n(\phi) = \phi^2$ meets the required hypotheses in Theorem \ref{ASthm}. Given the graph of $m_\beta(k)$ in Figure \ref{fig:disp}, it is easily believed---and even true---that $l(k) = -m_\beta(k)$ meets all conditions (a)-(c) when $\beta > 1/3$. Thus we get the conclusions of the theorem. Unwinding the very simple rescalings gives us most of Corollary \ref{depression cor}.
%\begin{cor}
%For all $\beta > 1/3$ and all $\ep$ sufficiently close to zero, there exists a solution\footnote{Theorem \ref{ASthm} only tells us that the solutions are in $E^1_0$. In this problem, however, we have the additional information that $m_\beta(k)$ grows like $|k|^{1/2}$ for large $|k|.$ With this, a straightforward bootstrapping argument demonstrates that the solutions are smooth. We omit these details.} $(w,c) \in \left(\cap_{r\ge0} E^r_0\right) \times \R$ solution of \eqref{TWE1} of the form 
%$$
%w(x) = \ep^2 W_\ep(\ep x) \mand
%c = 1 - \left( {3\beta -1 \over 6}  \right) \ep^2
%$$
%where
%$$
%W_\ep(X) = -{3\beta -1\over 4} \sech^2\left( X \over 2 \right) + R_\ep(X).
%$$
%The function $R_\ep$ satisfies, for any $r \ge 0$,  the estimate $\| R_\ep\|_{r,0} \le C_r \ep^2$.
%
%
%\end{cor}
Note that Theorem \ref{ASthm} only tells us that the solutions are in $E^1_0$
whereas Corollary \ref{depression cor} tells us they are smooth. In this problem, however, we have the additional information that $m_\beta(k)$ grows like $|k|^{1/2}$ for large $|k|$ and hence $\M_\beta$ is ``like" $\partial_x^{1/2}$. With this, a straightforward bootstrapping argument demonstrates that the solutions are smooth. We omit these details.

\section{Generalized solitary waves when $\beta \in (0,1/3)$.}\label{the general}

In this section we prove Theorem \ref{main theorem}. Throughout we fix $\beta \in (0,1/3)$ and we will, for the most part, not track how quantities depend on this quantity.\footnote{Our results hold for any such choice of $\beta$ but we make no claims upon how they depend on $\beta$ and in particular we make no claims about what happens at $\beta  \to 0^+$ or $\beta  \to 1/3^{-1}$.}

\subsection{A necessary solvability condition}\label{S:linearsolve1}

We begin our proof of Theorem \ref{main theorem} by doing something that is doomed to fail.  Nevertheless, we believe that understanding
the mechanism behind this failure is an important step in the journey to the proof of Theorem \ref{main theorem}.  Throughout, $r\geq 1$ (a regularity index) is fixed  but arbitrary and
$q>0$ (a decay rate) is taken to be sufficiently small.

To this end, we  first attempt to construct solutions of the nonlocal profile equation \eqref{TWE2} for $0<\eps\ll 1$ of the form
%Let us do a simple thing and suppose that, for $\ep > 0$, we may write the solution of \eqref{TWE2} we seek as
\be\label{ansatz}
W_\eps(X) = \sigma_\beta(X) + R(X),
\ee
where $R=R(\cdot;\eps)$ is to be some small, smooth, even\footnote{The equation respects even symmetry and as such we are
free to act, now and henceforth, that all functions are even.} function in $E^r_q$.
Inserting this ansatz into \eqref{TWE2} leads to the following equation for $R$:
%If we put this into \eqref{TWE2} we get the following equation for $R$:
\be\label{R eqn}
{\L_\ep} R + 2 \sigma_\beta R = \J_0 +\J_1
\ee
Here
\[
\L_\ep :=\ep^{-2}(\M_\beta^\ep-1 -\gamma_\beta \ep^2),\quad
\J_0=- \sigma_\beta^2 - \ep^{-2} (\M_\beta^\ep - 1 - \gamma_\beta\ep^2) \sigma_\beta \mand \J_1:=-R^2.
\]
Note that $\eps^2\mathcal{L}_\eps$ is simply the linearization of \eqref{TWE2} about the trivial solution $u=0$.
$\J_1$ is obviously quadratic in the unknown $R$, thus we obtain the following estimates 
via Sobolev embedding when $r \ge 1$:
\be\label{J1 estimates}
\| \J_1\|_{r,q} \le C_r  \|R\|_{r,q}^2 \mand
\| \J_1-\tilde{\J}_1\|_{r,q} \le C_r \left( \|R\|_{r,q} + \| \tilde{R}\|_{r,q} \right)\| R - \tilde{R}\|_{r,q}.
\ee
Above, by $\tilde{\J}_1$ we simply mean the quantity $\J_1$ evaluated at a function $\tilde{R}$ instead of at $R$.  

As for $\J_0$, it is a small forcing term.
From its definition 
 and that fact that $\sigma_\beta(X)$ solves \eqref{KDVTWE}
we see that
\be\label{J0 revisited}\begin{split}
\J_0 = &\gamma_\beta(\sigma_\beta'' -\sigma_\beta) - \ep^{-2} \left(\M_\beta^\ep -1 - \gamma_\beta \ep^2 \right)\sigma_\beta \\
        =& -\ep^{-2} ( \M_\beta^\ep-1-\gamma_\beta\ep^2 \partial_X^2)\sigma_\beta.
\end{split}
\ee
%
%
%\[
%-\J_0=\mathcal{L}_\eps\sigma_\beta+\sigma_\beta^2=\eps^{-2}\left[\eps^2\left(\gamma_\beta\sigma_\beta''-\gamma_\beta\sigma_\beta+\mathcal{O}(\eps^2)\right)\right]+\sigma_\beta^2
%\]
Then the formal expansion of $\M_\beta^\ep$ in \eqref{expando2} indicates that
$\J_0 \sim \ep^2 \partial_X^4 \sigma_\beta$. This argument can be made rigorous by way of Fourier analysis:
\begin{lemma}\label{J0 lemma}
There exists $q_0>0$ so that for any $q \in [0,q_0]$ and $r \ge 0$ we have 
\be\label{J0 estimate}
\| \J_0 \|_{r,q} \le C_r \ep^2.
\ee
\end{lemma}
The proof is in the Appendix.

%
%Both $\J_0$ and $\J_1$ are ``small." 
%The term $\J_0$ is small because $\sigma_\beta(X)$ is analytic, decaying and solves \eqref{KDVTWE}. 
%These allow us to make 
%a rigorous estimate out of the formal expansion \eqref{expando}. Specifically we have
%\be\label{J0 estimate}
%\| \J_0 \|_{r,q} \le C_r \ep^2.
%\ee
%We do not provide the details, but they can obtained, for instance, by modifying the proof of Lemma A.10 from \cite{faver-wright}.

%
%As for $\J_1$, it is manifestly quadratic and thus we obtain the following estimates via Sobolev embedding when $r \ge 1$:
%\be\label{J1 estimates}
%\| \J_1\|_{r,q} \le C_r  \|R\|_{r,q}^2 \mand
%\| \J_1-\tilde{\J}_1\|_{r,q} \le C_r \left( \|R\|_{r,q} + \| \tilde{R}\|_{r,q} \right)\| R - \tilde{R}\|_{r,q}.
%\ee
%Here by $\tilde{\J}_1$ we simply mean $\J_1$ evaluated at $\tilde{R}$ instead of at $R$.

%
%\begin{figure}[t]
%\centering
%%    \includegraphics[width=4in]{mp2.pdf}
%   \includegraphics[scale=0.5]{mp2.pdf}
% \caption{  \it Sketch of $l_\ep(K)$ vs $K$.
% }\label{m2}
%\end{figure}

To attempt to solve the nonlinear problem \eqref{R eqn}, we first consider the solvability of the nonhomogeneous linear equation\footnote{That is,
for each $F\in E^r_q$ we try to show there exists a unique solution $R\in E^{r+1/2}_q$ of \eqref{linsolve1}, and that this solution depends continuously
on $F$.}
\begin{equation}\label{linsolve1}
\mathcal{L}_\eps R+2\sigma_\beta R=F
\end{equation}
for $F\in E^r_q$.  Assuming the continuous solvability of \eqref{linsolve1} on $E^r_q$, we could then attempt to solve the nonlinear equation \eqref{R eqn}
through iteration.  

Recall from \eqref{critical freq} that for $\beta \in (0,1/3)$ and $c>1$ there exists unique $k_{\beta,c}>0$ such that $m_\beta(\pm k_{\beta,c}) = c$.
This implies that the symbol 
\be\label{lep}
l_\ep(K) : = \ep^{-2} \left(m_\beta(\ep K) -1 -\gamma_\beta \ep^2\right)
\ee
associated to the linear operator $\mathcal{L}_\eps$ satisfies
\be\label{kep prop}
l_\ep(\pm K_\ep) = 0
\ee
where
\be\label{this is kep}
K_\ep := {k_{\beta,1+\gamma_\beta \ep^2} \over\ep}.
\ee
From the above considerations it is easy to conclude that $K_\eps=\mathcal{O}(1/\eps)$\footnote{Which is to say that there are constants $0<k_1<k_2$
 such that
%\bes%\label{K is big}
$k_1 \ep^{-1} \le K_\ep \le k_2 \ep^{-1}$
%\ees
for $\ep$ close enough to zero. %{\bf We can tighten this if it helps in the stability part you are planning on writing.}
}
and is the unique frequency for which \eqref{kep prop} occurs.
%
%it is straightforward to show that for each $\beta\in(0,1/3)$ there exists a unique, positive frequency 
%$K_\eps=\mathcal{O}(1/\eps)$\footnote{Which is to say that there are constants $0<k_1<k_2$
% such that
%%\bes%\label{K is big}
%$k_1 \ep^{-1} \le K_\ep \le k_2 \ep^{-1}$
%%\ees
%for $\ep$ close enough to zero.}: see Figure \ref{m2} for a sketch of $l_\ep(K)$. 
%such that 
%\be\label{this is kep}
%l_\ep(\pm K_\ep) = 0:
%\ee
%see Figure \ref{m2} for a sketch of $l_\ep(K)$.  

Taking the Fourier transform of \eqref{linsolve1} and evaluating at $K_\eps$ implies that \eqref{linsolve1} is only solvable provided
that $R$ and the forcing $F$ satisfy $2\widehat{\sigma_\beta R}(K_\eps)=\widehat{F}(K_\eps)$.  For a generic $F\in E^r_q$, it follows that the single unknown $R$ is required to solve
two equations, hence the linear  problem \eqref{linsolve1} is  overdetermined.  Consequently, the above method of constructing a localized solution
of \eqref{TWE2} of the form \eqref{ansatz} fails.  

\subsection{Beale's method}\label{s:beal}

%To begin, we recall that 
Beale encountered nearly the same obstacle encountered in Section \ref{S:linearsolve1} in his work on the full gravity-capillary water wave problem \cite{Beal91}.  
In his investigation, he made the remarkable observation that
just as the special frequency $K_\eps$ causes difficulties at the linear level, it also points to a  way out.  
Indeed, observe that the lack of solvability of the linear forced
equation \eqref{linsolve1} stems from \eqref{kep prop} which, when written on the spatial side, simply states
that the linear problem $\L_\ep P = 0$ has a solution of the form
$
P = \cos(K_\ep X).
$
%from the resonance between the long wave limit $\sigma_\beta$
%and a mode with  Fourier frequency $K_\eps$, that is with 
%$\cos(K_\eps X)$.  
Beale used this observation to motivate a refinement of the ansatz \eqref{ansatz} that incorporates
a family of \emph{small amplitude, nonlinear} periodic traveling waves associated to the governing profile equation which
are roughly given by $a \cos(K_\ep X)$ where $|a| \ll 1$.
%whose profile and frequency  essentially limit to
%of the $\cos(K_\eps X)$ state at zero amplitude.  
By using the amplitude of this oscillation as an additional free variable, he was able to overcome the above difficulties.

In order to adapt Beale's method to the present case, we begin by recalling that, for each fixed $0<\eps\ll 1$, the nonlinear profile equation \eqref{TWE2} admits a family of small amplitude,
spatially periodic solutions with frequencies close to $K_\eps$.  Indeed, the following result follows from the analysis of \cite{HJ15b}:
%
% In particular, equation \eqref{TWE2} admits a family of spatially periodic solutions
%whose frequencies are close to $K_\ep$. We have the following result, whose proof can be found in [Mat/Vera]:

\begin{theorem}\label{MV}Fix $\beta \in (0,1/3)$.
There exists $\ep_0 > 0$, $\alpha_0>0$ and a mapping
\be\label{maps}
\begin{split}
%[-\alpha_0,\alpha_0] \times [-\ep_0,\ep_0]  &\longrightarrow \R \times C_\per^\infty  \\
%       (a,\ep) &\longmapsto  (K_\ep^a,\phi_\ep^a) \\
[-\alpha_0,\alpha_0] \times (0,\ep_0]  &\longrightarrow \R \times C_\per^\infty  \\
       (a,\ep) &\longmapsto  (K_\ep^a,\phi_\ep^a) \\
\end{split}
\ee
with the following properties:
\begin{itemize}
\item For each $\eps\in(0,\eps_0]$, the function $W(X)= a \phi_\ep^a(K_\ep^a X)$ solves \eqref{TWE2} for all $|a| \le \alpha_0$.
%\item $\Phi_\ep^a(X)$ is periodic in $X$ with frequency $K_\ep^a$.
\item $\phi_\ep^0(X) = \cos(X)$ and $K_\ep^0 = K_\ep$.
\item There exists $C>0$ such that $0 < \ep \le \ep_0$ and $|a|,|{\tilde{a}}| \le \alpha_0$ imply
\be\label{Klip}
|K_\ep^a - K_\ep^{{\tilde{a}}}| \le C|a - {\tilde{a}}|.
\ee
\item For all $r \ge 0$ there exists $C_r>0$ such that $0 < \ep \le \ep_0$ and $|a|,|{\tilde{a}}| \le \alpha_0$  imply
\be\label{philip}
\| \phi^a_\ep-\phi^{{\tilde{a}}}_\ep\|_{H^r_\per} \le C_r |a-{\tilde{a}}|.
\ee
\end{itemize}
\end{theorem}

Following Beale, we refine the ansatz \eqref{ansatz} by including one of the above small amplitude waves.  Specifically,
introducing the notation
\be\label{this is PHI}
{\Phi_\ep^a(X)}:= \phi_\ep^a(K_\ep^a X)
\ee
we attempt to construct solutions of the profile equation \eqref{TWE2} for $0<\eps\ll 1$ of the form
\be\label{bansatz}
W_\eps(X) = \sigma_\beta(X) + a {\Phi_\ep^a(X)} + R(X)
\end{equation}
where now both $R \in E^r_q$ and $a \in \R$ are unknowns.
% Then we will use $a$---which we interpret as the amplitude of the periodic part---as our additional variable.
Inserting the refined ansatz \eqref{bansatz} into \eqref{TWE2}  gives the equation
\bes
\L_\ep  R + 2  \sigma_\beta R + 2 a \sigma_\beta \Phi_\ep^a+2a \Phi^a_\eps R = \J_0 +  \J_1
\ees
where here $\mathcal{L}_\ep$, $\J_0$ and $\J_1$ are as before.  Note that the term $a\Phi^a_\eps R$ is clearly nonlinear 
in the unknowns $a$ and $R$.  The term $a\sigma_\beta\Phi^a_\eps$ however has an $\mathcal{O}(a)$ term coming from the fact
that $\Phi^0_\eps(X)=\cos(K_\eps X)$.  Incorporating this additional linear term on the left hand side of leads to the nonlinear equation
\be\label{our shark}
\L_\ep  R + 2  \sigma_\beta R + 2 a \sigma_\beta \Phi_\ep^0= \J_0 +  \J_1+\J_2+\J_3
\ee
where here we have
\[
\J_2:=-2 a \sigma_\beta \left( \Phi_\ep^a -\Phi_\ep^0\right) \mand \J_3  = -2 a  \Phi^a_\ep R.
\]
%$\L_\ep$, $\J_0$ and $\J_1$ are unchanged from above.

Given \eqref{Klip} and \eqref{philip}  we see that $\J_2$ is nonlinear in the sense that it is morally $\O(a^2)$.
This leads to: 
\begin{lemma}\label{J2 lemma} There exists $\ep_0>0$ and $q_0>0$ such that, for all $r \ge 0$, $q \in [0,q_0]$ and $\ep \in (0,\ep_0]$ 
we have
\be\label{J2 estimates}
\| \J_2\|_{r,q} \le C_r \ep^{-r} a^2
\mand 
\| \J_2-\tilde{\J}_2\|_{r,q} \le C_r \ep^{-r} (|a|+|\tilde{a}|)|a-\tilde{a}|.
\ee
\end{lemma}
As for $\J_3$, it is roughly $\O(aR)$. Specifically:
\begin{lemma}\label{J3 lemma} There exists $\ep_0>0$ and $q_0>0$ such that, for all $r \ge 0$, $q \in [0,q_0]$ and $\ep \in (0,\ep_0]$ 
we have
\begin{equation}\label{J3 estimates}
%\begin{multline}\label{J3 estimates}
\left\{\begin{aligned}
\| \J_3\|_{r,q} &\le C_r \ep^{-r} |a|\|R\|_{r,q}\\
&\hspace{-2em}\mand
\| \J_3-\tilde{\J}_3\|_{r,0} \le C_{r,q} \ep^{-r} \left( (\|R\|_{r,q}+\|\tilde{R}\|_{r,q})|a-\tilde{a}| +  (|a|+|\tilde{a}|)\|R-\tilde{R}\|_{r,0}\right).
\end{aligned}\right.
\end{equation}
\end{lemma}
%\end{multline}
In both \eqref{J2 estimates} and \eqref{J3 estimates} above, $\tilde{\J}_l$ simply represents the quantity $\J_l$ evaluated at $\tilde{R}$ and $\tilde{a}$.
An important feature of these estimates is that there is a mismatch in the decay rates of the pieces in 
the estimate in \eqref{J3 estimates}(ii): specifically, on the left we measure in $H^r_0$ but the right requires $R$ and $\tilde{R}$ to be in
$H^r_q$ with $q > 0$.  In particular, the constant $C_{r,q}$ diverges as $q \to 0^+$.
%To prove these estimate follows the same logic as was used to prove the estimates in Appendix E.4 of \cite{hoffman-wright}.
We provide the justification for Lemmas \ref{J2 lemma} and \ref{J3 lemma} in the Appendix.

Our goal is now to resolve the nonlinear equations \eqref{our shark} for $R$ and $a$.  
As in Section \ref{S:linearsolve1}, we will proceed by first considering the solvability
of the associated nonhomogeneous linear equation.  After we have shown that  this can be continuously solved, we solve the full nonlinear
equation \eqref{our shark} through iteration.  %This program is carried out in the remaining sections of Section \ref{S:proof}.

\subsection{The linear problem}
The left hand side of \eqref{our shark} is linear in $R$ and $a$. 
We claim it is a bijection in an appropriate sense. Specifically we have the following linear solvability result.

\begin{proposition} \label{linear mover}
There exists $\ep_0>0$ and $q_0>0$ for which the following hold when 
$\ep \in (0,\ep_0]$, $q \in (0,q_0]$ and $r \ge 0$. There are linear maps 
\[
\Rc_\ep : E^{r}_q \longrightarrow E^{r+1/2}_q  \mand \A_\ep :E^{r}_q \longrightarrow \R
\]
such that 
\be\label{linear shark}
\L_\ep  R + 2  \sigma_\beta R + 2 a \sigma_\beta \Phi_\ep^0 = G \in E^{r}_q
\ee
if and only if
\[
R = \Rc_\ep G \mand a = \A_\ep G.
\]
Moreover these maps are continuous and satisfy the estimates
\[
\|\Rc_\ep\|_{B(E^r_q,E^{r+1/2}_q)} \le C_{r,q} \mand \|\A_\ep\|_{B(E^r_q,\R)} \le C_{r,q} \ep^r.
\]
\end{proposition}

\begin{remark}
It is important to note that the size of $\mathcal{A}_\eps$ is directly related to the (Sobolev) smoothness of the forcing function $G$.
This observation will be important in our coming work.
\end{remark}

\begin{proof}
%Our goal is to solve \eqref{linear shark} for $R$ and $a$ given $G$.
Recalling \eqref{kep prop}, we first see that to solve \eqref{linear shark} requires the linear solvability condition
\be\label{solve linear}
 2  \hat{\sigma_\beta R}(K_\ep) + 2 a \hat{\sigma_\beta \Phi_\ep^0}(K_\ep) = \hat{G}(K_\ep)
\ee
to hold.  Since $\Phi_\ep^0(X) = \cos(K_\ep X)$ by Theorem \ref{MV}, we can calculate 
\be\label{kappa}\begin{split}
2\hat{\sigma_\beta \Phi_\ep^0}(K_\ep)&={1 \over \pi} \int_\R \sigma_\beta(X) \cos(K_\ep X) e^{-i K_\ep X} dX\\
 &={1 \over \pi} \int_\R \left({1 \over 2} \sigma_\beta(X) + {1 \over 2} \sigma_\beta(X) e^{-2iK_\ep X}\right) dX\\
 &=  \hat{\sigma_\beta}(0) +  \hat{\sigma_\beta}(2K_\eps)\\&=:\chi_\ep
\end{split}\ee
Note that since $\sigma_\beta\in L^1(\RM)$ is positive, we know $\hat\sigma_\beta(0)>0$.
Moreover, the analyticity of $\sigma_\beta$ and the fact that $K_\ep = \O(1/\ep)$ implies $|\hat{\sigma_\beta}(2K_\ep)|$ is exponentially small in $\ep$.  
Consequently  $\chi_\ep$ and $\chi_\ep^{-1}$ are bounded uniformly in $\ep$ for $0<\eps\ll 1$.
It follows that we can solve \eqref{solve linear} for $a$ explicitly in terms of $G$ and $R$ as
\be\label{this is a}
a = \chi_\ep^{-1} \left(\hat{G}(K_\ep) - 2 \hat{\sigma_\beta R}(K_\ep)\right),
\ee
thus guaranteeing that this choice of $a$ ensures the linear solvability \eqref{solve linear} holds for a given $G\in E^r_q$, provided
that we can now resolve \eqref{linear shark} for $R$.

Before substituting \eqref{this is a} into \eqref{linear shark},  define the operator $\P_\eps:E^r_q\to E^r_q$ by
\[
\P_\ep F = F-2\hat{F}(K_\ep)\chi_\ep^{-1} \sigma_\beta \Phi_\ep^0.
\]
By construction, for all  $F\in E^r_q$ 
\be\label{pi prop}
\hat{\P_\ep F}(K_\ep) = 0.
\ee% \mand \P_\ep^2 = \P_\ep.$$ 
Furthermore
  \be\label{Pep estimate}
\|\P_\ep\|_{B(E^r_q)} \le C_{r,q}.
\ee
To prove this, one needs the Riemann-Lebesgue estimate
\be\label{RL}
\left \vert \hat{F}(K) \right \vert \le C K^{-r} q^{-1/2} \| F\|_{r,q}
\ee
which holds for any $\pm K > 1$ and $q > 0$.  See Lemma A.5 in \cite{Faver-Wright} for a proof.
With this, \eqref{Pep estimate} follows  quickly from the fact that $\Phi_\ep^0(X) = \cos(K_\ep X)$
and that $\chi_\ep^{-1}$ is bounded above.
Specifically
\bes\begin{split}
\| \P_\ep F\|_{r,q} &\le \| F\|_{r,q} + 2\chi_\ep^{-1} \left \vert \hat{F}(K_\ep) \right \vert \| \sigma_\beta \Phi_\ep^0\|_{r,q}\\
&\le \left(1 +C_q K_\ep^{-r}  \| \sigma_\beta \|_{r,q} \| \Phi_\ep^0\|_{W^{r,\infty}(\R)}\right) \| F\|_{r,q} \\
&\le \left(1 +C_q K_\ep^{-r}  \| \sigma_\beta \|_{r,q}K_\ep^ r\right) \| F\|_{r,q} \\
& \le C_{r,q}  \| F\|_{r,q}
\end{split}
\ees

Substituting \eqref{this is a} into \eqref{linear shark} gives the equation
\be\label{ls 2}
\L_\ep R + \P_\ep (2 \sigma_\beta R) = \P_\ep G.
\ee
Given \eqref{pi prop}, the linear solvability condition coming from \eqref{kep prop} is satisfied in \eqref{ls 2}.  The next result 
shows that this the operator $\mathcal{L}_\eps$ is indeed continuously invertible on the range of $\P_\eps$.
%
%At this point we have met the solvability condition \eqref{solve linear}. It happens that \eqref{solve linear} is not only
%necessary but sufficient to solve equations of the form $\L_\ep R = F \in E^r_q$.  We have:

\begin{lemma}\label{suff}
There exists $q_0>0$ and $\ep_0>0$ such the following holds for all $q \in (0,q_0]$, $\ep \in (0,\ep_0]$ and $r \ge 0$.
Suppose that $F \in E^r_q$ and 
$\hat{F}(K_\ep) =0$. Then there exists a unique $R \in E^{r+1/2}_q$, which we denote by $\L_\ep^{-1} F$, such that $\L_\ep R = F$. 
Finally, we have 
$$
\| R\|_{r+1/2,q} = \| \L_\ep^{-1} F\|_{r+1/2,q} \le C_{r,q} \| F\|_{r,q}.
$$
\end{lemma}

We provide the proof in the Appendix.  Together with \eqref{pi prop}, Lemma \ref{suff} allows us to rewrite \eqref{ls 2} as
\be\label{ls 4}
R + \L_\ep^{-1}\P_\ep (2 \sigma_\beta R) = \L_\ep^{-1} \P_\ep G.
\ee
At first glance, it is not entirely clear that we have made progress towards our goal of solving for $R$, as we have to figure out how to invert
the operator $1 + \L_\ep^{-1}\P_\ep(2 \sigma_\beta \cdot)$.
To make this step requires a critical feature of $\L_\ep^{-1} \P_\ep$: it is small perturbation of $-\gamma_\beta^{-1} (1-\partial_X^2)^{-1}$.
Specifically, if we put
\[
\G_\ep := \ep^{-1} \left(\L_\ep^{-1} \P_\ep+\gamma_\beta^{-1} (1-\partial_X^2)^{-1}\right)
\]
then we  can establish the following result.

\begin{lemma}\label{desing lemma}
There exists $q_0>0$ and $\ep_0>0$ such that for all $q \in (0,q_0]$, $\ep \in (0,\ep_0]$ and $r \ge 0$ we have
$$\| \G_\ep\|_{B(E^r_q,E^r_q)} \le C_{r,q}.$$
\end{lemma}

Again, we provide the proof of Lemma \ref{desing lemma} in the Appendix.  Notice, however that this result is not unexpected 
since the formal expansion \eqref{expando2} indicates
\be\label{formal L}
\L_\ep = -\gamma_\beta (1-\partial_X^2) + \O(\ep^2).
\ee
Lemma \ref{desing lemma} provides a meaningful and rigorous version of \eqref{formal L} and as such represents one of the keys of our analysis.
Indeed, as in our discussion directly above the statement of Theorem \ref{main theorem}, the $\mathcal{L}_\eps$ is a singular perturbation of the operator $-\gamma_\beta(1+\partial_X^2)$,
and resolving this singular limit is one of the main technical difficulties faced in the present study.   

We can rewrite \eqref{ls 3} as
\be\label{ls 3}
\bunderbrace{R -\gamma_\beta^{-1} (1-\partial_X^2)^{-1} (2 \sigma_\beta R) + 2\ep \G_\ep (\sigma_\beta R)}{\S_\ep R}= \L_\ep^{-1} \P_\ep G.
\ee
%Now put 
%$$
%\S_\ep := 1 -\gamma_\beta^{-1} (1-\partial_X^2)^{-1} (2 \sigma_\beta \cdot)
%+ 2\ep \G_\ep (\sigma_\beta \cdot).$$
Using \eqref{desing lemma}, we see that $\S_0 =  1 -\gamma_\beta^{-1} (1-\partial_X^2)^{-1} (2 \sigma_\beta \cdot)$, which is recognized as as $-(1-\partial_X^2)^{-1}$
applied to the linearization of the KdV profile equation \eqref{KDVTWE} about the KdV solitary wave $\sigma_\beta$.  This latter operator has been very well studied in the literature,
and in particular it follows from standard Sturm-Liouville theory on $L^2(\RM)$ that ${\rm ker}(S_0)={\rm span}\{\sigma_\beta'\}$.  Since $\sigma_\beta'$ is odd by construction,
it follows that $\S_0$ is invertible on the class of even functions in $E^r_0(\RM)$ for any $r>0$.  Following (for instance) Appendix D.10 of \cite{Faver}, this observation can be extended to the weighted
space $E^r_q$ via operator conjugation. Specifically:

\begin{lemma}\label{sturm}
There exists $q_0 >0$ such that for all $r \ge 0$ and $q \in [0,q_0]$ the operator $\S_0$
is a bounded and invertible map from $E^r_q \to E^r_q$. In particular 
\[
  \|\S_0\|_{B(E^r_q )} + \|\S_0^{-1}\|_{B(E^r_q)} \le C_{r}.
\]
\end{lemma}

Notice that Lemma \ref{desing lemma} implies that $\S_\ep$ is a small regular perturbation of $\S_0$ and thus, for $\ep$ sufficiently small
the operator $\S_\eps$ is invertible on $E^r_q$ as well and satisfies the estimate
%it is invertible. The usual rigamarole tells us that
\[
  \|\S_\ep\|_{B(E^r_q )} + \|\S_\ep^{-1}\|_{B(E^r_q)} \le C_r.
\]
Consequently, we can invert $\S_\eps$ in \eqref{ls 3} and solve for $R$ completely in terms of $G\in E^r_q$ as
\be\label{ls 5}
R =\Rc_\ep G:= \S_\ep^{-1}\L_\ep^{-1} \P_\ep G.
\ee
Following through the above estimates we have we find that
\[
\| \Rc_\ep\|_{B(E^{r}_q,E^{r+1/2}_q)}  \le C_{r,q}.
\]

Furthermore, returning to \eqref{this is a} we find that $a$ is given in terms of $G$ as
\[
a =\A_\ep G:= \chi_\ep^{-1} \left( \hat{G}(K_\ep)  - 2 \hat{\sigma_\beta \Rc_\ep G}(K_\ep)\right).
\]
To estimate the size of $a$ in terms of $\eps$, we use \eqref{RL} and the fact that $K_\ep = \O(1/\ep)$ 
%
 %observe that for $G\in E^r_q$ we have by simple integration by parts the identity
%\begin{align*}
%\hat{G}(K_\eps)%&=\frac{1}{2\pi} \int_{\RM} G(x)e^{-iK_\eps x}dx
%=\frac{1}{2\pi}\int_{\RM}G(x)\frac{1}{i^rK_\eps^r}\frac{d^r}{dx^r}\left(e^{-iK_\eps x}\right)dx
%=\frac{(-1)^r i^r}{2\pi K_\eps^r}\int_{\RM}G^{(r)}(x)e^{-iK_\eps x}dx.
%\end{align*}
%Recalling that $K_\eps=\mathcal{O}(1/\eps)$ it follows by Cauchy-Schwarz that
%\begin{align*}
%\left|\hat{G}(K_\eps)\right|%&\leq C_r\eps^{r}\int_{\RM}\left| G^{(r)}(x)\right|\cosh^{q/2}(x)\sech^{q/2(x)}dx\\
%&\leq C_{r}\eps^r\left(\int_{\RM}\sech^q(x)dx\right)^{1/2}\|G\|_{r,q}=C_{r,q}\eps^r\|G\|_{r,q}.
%\end{align*}
to get:
\[
 \|\A_\ep \| _{B(E^r_q,\R)} \le C_{r,q} \ep^r.
\]
This completes the proof of  Proposition \ref{linear mover}.
\end{proof}

\subsection{Nonlinear solvability}

We now return to constructing a solution of the form \eqref{bansatz} to the nonlinear equation \eqref{our shark}.  Thanks to Proposition \ref{linear mover} 
we see that solving \eqref{our shark} is equivalent to solving the fixed point problem
\begin{equation}\label{jaws 2}\begin{split}
R = \Rc_\ep \left(\J_0 + \J_1 + \J_2 + \J_3 \right)=:\N^1_\ep(R,a)\\
a = \A_\ep \left(\J_0 + \J_1 + \J_2 + \J_3 \right)=:\N^2_\ep(R,a).
\end{split}
\end{equation}
on the space $E^r_q\times\RM$, where here the $\J_i$ are defined as in Section \ref{s:beal} above.  The goal is to show that in a sufficiently small neighborhood of $E^r_q\times\RM$, 
the nonlinear system \eqref{jaws 2} has a unique solution.  To this end, we begin by collecting necessary estimates on the nonlinear terms, all of which follow in a direct way by
Proposition \ref{linear mover} and \eqref{J1 estimates}, \eqref{J0 estimate}, \eqref{J2 estimates}, and \eqref{J3 estimates}.
%
%The following proposition, which follows in a direct way  Proposition \ref{linear mover} and \eqref{J0 estimate}, \eqref{J1 estimates}, \eqref{J2 estimates}, \eqref{J3 estimates}, collects all the estimates we will need to show that a solution of \eqref{jaws 2} exists:

\begin{proposition}\label{work it}
There exist $q_*>0$ and $\ep_0>0$ with the following properties.
For all $r \ge 1$ there exist $\kappa_r>0$ such that $\ep \in (0,\ep_0]$ implies
\be\label{N1 map}
\|\N_\ep^1(R,a)\|_{r,q_*} \le \kappa_r\left(\ep^2 +
\ep^{-r+1/2} a^2 + \ep^{-r+1/2} |a| \|R\|_{r-1/2,q_*} + \| R\|^2_{r-1/2,q_*} \right),
\ee
\be\label{N2 map}
|\N_\ep^2(R,a)| \le \kappa_r\left(\ep^{r+2} +
a^2 +  |a| \|R\|_{r,q_*} + \ep^r \| R\|^2_{r,q_*} \right),
\ee
\begin{multline}\label{N1 lip}
\|\N_\ep^1(R,a)-\N_\ep^1(\tilde{R},{\tilde{a}})\|_{r,0} \le 
 \kappa_r \ep^{-r} (\|R\|_{r,q_*}+\|\tilde{R}\|_{r,q_*}+|a|+|{\tilde{a}}|)|a-{\tilde{a}}|\\
+
\kappa_r\left( \|R\|_{r,0} + \| \tilde{R}\|_{r,0} + \ep^{-r}\left(|a|+|{\tilde{a}}|\right) \right)\| R - \tilde{R}\|_{r,0}
%& +\kappa_{r,q_*} \ep^{-r+1/2} \left( |a-{\tilde{a}}|\|R\|_{r-1/2,q_*} +  |a|\|R-\tilde{R}\|_{r-1/2,0}\right).
\end{multline}
and
\begin{multline}\label{N2 lip}
|\N_\ep^2(R,a)-\N_\ep^2(\tilde{R},{\tilde{a}})|  \le  \kappa_r  (\|R\|_{r,q_*} +\|\tilde{R}\|_{r,q_*} +|a|+|{\tilde{a}}|)|a-{\tilde{a}}|\\
+\kappa_r\ep^{r} \left( \|R\|_{r,q_*} + \| \tilde{R}\|_{r,q_*}+\ep^{-r} \left(|a|+|{\tilde{a}}|\right)  \right)\| R - \tilde{R}\|_{r,0}.
%& +\kappa_{r,q_*} \left( |a-{\tilde{a}}|\|R\|_{r,q_*} +  |a|\|R-\tilde{R}\|_{r,0}\right).
\end{multline}
\end{proposition}

Equipped with the above estimates, we construct our small solution to \eqref{jaws 2} by a classical iterative argument. 
To begin, let $R_0:=0$, $a_0 :=0$ and for $n\ge0$ define
\be\label{the scheme}
R_{n+1}:=\N_\ep^1(R_{n},a_{n}) \mand a_{n+1}:=\N_\ep^2(R_{n},a_{n}).
\ee
We first claim that for each fixed $r\geq1$ there exists $\ep_{r} > 0$ such that
\be\label{ind hyp}
\| R_{n}\|_{r,q_*} \le 2 \kappa_r \ep^2 \mand  |a_n| \le 2 \kappa_r \ep^{r+2}.
\ee
for all $\ep \in (0,\ep_{r}]$.
Here $\kappa_r>0$ is as in Proposition \ref{work it}.
The proof is by induction and the base case is obvious. 
 For the inductive step, if we assume \eqref{ind hyp} then \eqref{the scheme} implies by way of the estimates \eqref{N1 map} and \eqref{N2 map} that
%$$
%\| R_{n+1}\|_{r,q_*} \le \kappa_r\left(\ep^2 +
%\ep^{-r+1/2} a_{n}^2 + \ep^{-r+1/2} |a_n| \|R_n\|_{r,q_*} + \| R_n\|^2_{r,q_*} \right)
%$$
%and
%$$
%|a_{n+1}| \le \kappa_r\left(\ep^{r+2} +
%a_n^2 +  |a_n| \|R_n\|_{r,q_*} + \ep^r \| R_n\|^2_{r,q_*} \right).
%$$
%Assuming \eqref{ind hyp} and $0 < \ep <1$, these imply:
\[
\| R_{n+1}\|_{r,q_*} \le \kappa_{r,q_*}\left(\ep^2 +
\ep^{-r+1/2} (2 \kappa_{r,q_*} \ep^{r+2})^2 + \ep^{-r+1/2} |2 \kappa_{r,q_*} \ep^{r+2}|2 \kappa_{r,q_*} \ep^2 + (2 \kappa_{r,q_*} \ep^2)^2 \right)
\]
and
\[
|a_{n+1}| 
\le \kappa_{r,q_*}
\left(
\ep^{r+2} +
(2 \kappa_{r,q_*} \ep^{r+2})^2 +  |2 \kappa_{r,q_*} \ep^{r+2}| 2 \kappa_{r,q_*} \ep^2+ \ep^r (2 \kappa_{r,q_*} \ep^2)^2 
\right).
\]
which, after tidying up and requiring $0 < \ep <1$, gives the estimates
\[
\| R_{n+1}\|_{r,q_*} \le \kappa_r\ep^2 \left(1 
+12 \kappa^2_{r}\ep^2 \right)
\mand
|a_{n+1}| 
\le \kappa_r \ep^{r+2}
\left(1
+ 12 \kappa_r^2 \ep^{2}
\right).
\]
By setting 
\be\label{ep condition}
\ep_{r} := {1\over \sqrt{12} \kappa_r},
\ee
then we have attained our goal of showing \eqref{ind hyp} for all $n$ when $\ep \in (0,\ep_r]$, with $\eps_r$ defined as in \eqref{ep condition}.

Thus, for $\ep$ small enough we have that
$
\left\{(R_n,a_n) \right\}_{n \ge 0}
$
a bounded sequence in $E^r_{q_*} \times \R$. 
%Since this is a Hilbert space, the sequence has a weakly convergent subsequence. Call the limit $(R_\ep,a_\ep)$. Lower semi-continuity for weak limits tells us that
%\be\label{weak bound}
%\| R_{\ep}\|_{r,q_*} \le 2 \kappa_{r,q_*} \ep^2 \mand  |a_\ep| \le 2 \kappa_{r,q_*} \ep^{r+2}.
%\eet
We now demonstrate that this sequence  converges strongly in the space $E^r_0 \times \R$.  Using \eqref{the scheme}
together with  the estimates \eqref{N1 lip} and \eqref{N2 lip} we find directly that
\begin{equation*}
\begin{split}
\| R_{n+1} - R_n\|_{r,0}=&\|\N_\ep^1(R_n,a_n)-\N_\ep^1(R_{n-1},a_{n-1})\|_{r,0}\\  \le
&\kappa_r \ep^{-r} (\|R_n\|_{r,q_*}+\|R_{n-1}\|_{r,q_*}+|a_n|+|a_{n-1}|)|a_n-a_{n-1}|\\
&+\kappa_r\left( \|R_n\|_{r,0} + \| R_{n-1}\|_{r,0} + \ep^{-r}\left(|a_n|+|a_{n-1}|\right) \right)\| R_n - R_{n-1}\|_{r,0}\\
 \end{split}
\end{equation*}
and
\begin{equation*}
\begin{split}
|a_{n+1} - a_n|= &|\N_\ep^2(R_{n},a_{n})-\N_\ep^2(R_{n-1},a_{n-1})| \\ \le &
 \kappa_r  (\|R_n\|_{r,q_*} +\|R_{n-1}\|_{r,q_*} +|a_n|+|a_{n-1}|)|a_n-a_{n-1}|\\
&+\kappa_r\ep^{r} \left( \|R_n\|_{r,q_*} + \| R_{n-1}\|_{r,q_*}+\ep^{-r} \left(|a_n|+|a_{n-1}|\right)  \right)\| R_n - R_{n-1}\|_{r,0}.
\end{split}
\end{equation*}
Using \eqref{ind hyp} it follows from the above estimates that 
\begin{equation*}
\begin{split}
\| R_{n+1} - R_n\|_{r,0}\le
&8 \kappa^2_{r} \ep^2\left( \ep^{-r}   |a_n-a_{n-1}|+\| R_n - R_{n-1}\|_{r,0}
  \right)\\
 \end{split}
\end{equation*}
and
\begin{equation*}
\begin{split}
|a_{n+1} - a_n| \le &
8 \kappa^2_{r}\ep^2 \left( |a_n-a_{n-1}|+ \ep^{r}\| R_n - R_{n-1}\|_{r,0}\right),
%& +\kappa_{r,q_*} \left( |a-{\tilde{a}}|\|R\|_{r,q_*} +  |a|\|R-\tilde{R}\|_{r,0}\right).
\end{split}
\end{equation*}
valid for all $n\geq 1$ and $\eps\in(0,\eps_r]$.  Taken together, it follows that
%These together imply that
\be\label{final estimate}
\| R_{n+1} - R_n\|_{r,0} + \ep^{-r}|{a}_{n+1} - a_n|  \le 8 \kappa^2_{r}\ep^2 \left(\| R_n - R_{n-1}\|_{r,0}+  \ep^{-r} |{a}_n-{a}_{n-1}|\right),
\ee
which, in turn, gives
\bes
\| R_{n+1} - R_n\|_{r,0} + \ep^{-r}|{a}_{n+1} - a_n|  \le \left(8 \kappa^2_{r}\ep^2\right)^n \left(\| R_1 - R_{0}\|_{r,0}+  \ep^{-r} |{a}_1-{a}_{0}|\right).
\ees
Using the initial conditions $R_0 = 0$, $a_0 = 0$ along with  \eqref{ind hyp} converts this to
\bes
\| R_{n+1} - R_n\|_{r,0} + \ep^{-r}|{a}_{n+1} - a_n|  \le  4\kappa_r \ep^2 \left(8 \kappa^2_{r}\ep^2\right)^n. 
\ees
By the triangle inequality, it now follows that for $m,n\geq 1$ and $m>n$ we have
\[
\| R_{m} - R_n\|_{r,0} + \ep^{-r}|{a}_{m} - a_n|  \le \sum_{k=n}^{m-1} 4\kappa_r \ep^2 \left(8 \kappa^2_{r}\ep^2\right)^k.
\]
Note that since since $\ep \in  (0,\ep_{r}]$, with $\ep_{r}$ as in \eqref{ep condition}, we have that $8 \kappa^2_{r}\ep^2<1$ and hence, using a geometric series, we find
\[
\| R_{m} - R_n\|_{r,0} + \ep^{-r}|{a}_{m} - a_n|  \le {4\kappa_r \ep^{2} \left(8 \kappa^2_{r}\ep^2\right)^n \over 1 - 8 \kappa^2_{r}\ep^2} .
\]
Since the right hand side converges to zero as $n \to \infty$, it follows that $\left\{(R_n,a_n) \right\}_{n \ge 0}$ is a Cauchy sequence in $E^r_0 \times \R$.
As $E^r_0\times\RM$ is clearly a Hilbert space, it follows that exists $(R_\ep,a_\ep) \in E^r_0 \times \R$ such that
\be\label{converge 1}
(R_n,a_n) \underset{E^r_0 \times \R} {\longrightarrow}(R_\ep,a_\ep)  \quad \text{as $n \to \infty$}.
\ee
%Given \eqref{ind hyp} we have
%\be\label{sol est}
%\| R_{\ep}\|_{r,0} \le 2 \kappa_r \ep^2 \mand  |a_\ep| \le 2 \kappa_r \ep^{r+2}.
%\ee

Our next goal is to show that, in fact, the function $R_\eps$ is exponentially localized.  To this end, let $q_*>0$ be fixed as in Proposition \ref{work it}
%The above argument shows $R_\ep \in E^r_0$. It turns out it is in $E^r_{q_*}$ too. We showed at \eqref{ind hyp} that $\left\{R_n \right\}_{n \ge 0}$
and note that \eqref{ind hyp} implies that $\{R_n\}_{n\geq 0}$ is a  bounded in $E^r_{q_*}$.   Since $E^r_{q_*}$ is a Hilbert space it follows
that we can extract a weakly convergent subsequence in $E^r_{q_*}$.  Denote this weak limit as $\tilde{R}_\ep \in E^r_{q_*}$ and note that 
since $E^r_{q_*} \subset E^r_0$, the same subsequence also converges weakly to  $\tilde{R}_\ep$ in the unweighted space $E^r_0$.  
However, since we have already shown that $R_n \to R_\ep$ strongly in $E^r_0$ we must
have $R_\ep = \tilde{R}_\ep$ by the uniqueness of weak limits which, in turn, immediately implies that $R_\ep \in E^r_{q_*}$, as claimed. % since $\tilde{R}_\ep$ is. 
Furthermore, since norms on Hilbert spaces are lower semi-continuous with respect to weak limits, we know from \eqref{ind hyp} that
\be\label{Rep est}
\| R_\ep \|_{r,q_*} \le 2 \kappa_r \ep^2 \mand  |a_\ep| \le 2 \kappa_r \ep^{r+2}.
\ee

The next step is to show that the pair $(R_\eps,a_\eps)\in E^r_q\times\RM$ is indeed a solution of the nonlinear system \eqref{jaws 2}, i.e. that
\be\label{victory}
R_\ep = \N_\ep^1(R_\ep,a_\ep)\mand a_\ep = \N_\ep^2(R_\ep,a_\ep).\ee 
Since we know each term in the sequence $\{ R_n \}_{n \ge 0}$ meets the bound in \eqref{ind hyp}, and since we know  $R_\ep$ 
satisfies the similar estimate \eqref{Rep est}, we can conclude from \eqref{N1 lip}, \eqref{N2 lip} and \eqref{converge 1} that
\[
(\N_\ep^1(R_n,a_n),\N_\ep^2(R_n,a_n))  \underset{E^r_0 \times \R} {\longrightarrow}(\N_\ep^1(R_\ep,a_\ep),\N_\ep^2(R_\ep,a_\ep)) \quad \text{as $n \to \infty$}.
\]
Consequently, we can pass to the limit in the iterative scheme \eqref{the scheme} to conclude \eqref{victory}.  This establishes the existence component of Theorem \ref{main theorem}.

It is also the case that $(R_\ep,a_\ep)$ is the unique solution of \eqref{jaws 2} which meets the estimates
in \eqref{Rep est}. Indeed, if there were another pair $(\tilde{R}_\ep,{\tilde{a}}_\ep)\in E^r_q\times\RM$ that satisfies both \eqref{jaws 2} and \eqref{Rep est}, then it is apparent that
\bes
\| R_{\ep} - \tilde{R}_\ep\|_{r,0} + \ep^{-r}|{a}_{\ep} - {\tilde{a}}_\ep| 
= \| \N_\ep^1(R_{\ep},a_\ep) - \N_\ep^1(\tilde{R}_\ep,{\tilde{a}}_\ep)\|_{r,0} + \ep^{-r}|\N_\ep^2(R_{\ep},a_\ep) - \N_\ep^2(\tilde{R}_\ep,{\tilde{a}}_\ep)|. 
% \le 8 \kappa^2_{r}\ep^2 \left(\| R_n - R_{n-1}\|_{r,0}+  \ep^{-r} |{a}_n-{a}_{n-1}|\right).
\ees
Using the same argument that led from \eqref{N1 lip} and \eqref{N2 lip} to \eqref{final estimate}, we find that
\bes
\| R_{\ep} - \tilde{R}_\ep\|_{r,0} + \ep^{-r}|{a}_{\ep} - {\tilde{a}}_\ep| 
 \le 8 \kappa^2_{r}\ep^2 \left(\| R_\ep - \tilde{R}_{\ep}\|_{r,0}+  \ep^{-r} |{a}_\ep-{{\tilde{a}}}_{\ep}|\right).
\ees
which, since $8 \kappa^2_{r}\ep^2<1$, implies $(\tilde{R}_\ep,{\tilde{a}}_\ep)=(R_\ep,a_\ep)$, as claimed.

It remains to discuss the smoothness of the solutions $R_\eps$ constructed above.
The smoothing property imputed in \eqref{N1 map} implies, since our solutions are fixed points, that
our functions $R_\ep$ are $C^\infty$ by a routine bootstrap argument. 
But there is more to the story.
Note that we left the precise value of the regularity index $r$
unspecified above; there was no restriction on its size. 
The uniqueness property together with the containment $E^{r+1}_{q_*} \subset E^{r}_{q_*}$
implies   that the solution we construct at order $r$ coincides with the ones we construct at  
higher regularity, at least for $\ep$ small enough. This is another avenue for establishing
the smoothness of the solutions, but also more than that.
Note that the larger the regularity index $r$, the tighter the bound on $a_\ep$ is in \eqref{Rep est}. 
Thus we can conclude that  $|a_\ep| \le C_r \ep^r$ for all $r \ge 1$, {\it i.e.} that it is small beyond all algebraic orders of $\ep$.
By putting $P_\ep(x) = a_\ep \Phi_\ep^{a_\ep}(x/\ep)$ we have proven Theorem \ref{main theorem}.

\section{Discussion on Stability}\label{S:stability}

In this final section, we briefly consider the spectral stability of the generalized solitary waves constructed in Theorem \ref{main theorem}.  Specifically, we are interested
in the ability of these generalized solitary waves to persist when subject to small perturbations.   As we will see, a necessary condition for our small generalized solitary waves
to be stable is that their oscillatory endstates be stable.  This is made rigorous below by the Weyl essential spectrum theorem.  The stability of the small, oscillatory endstates
has recently been investigated in \cite{HJ15b}, allowing us to make observations about the stability of the patterns constructed in Theorem \ref{main theorem}.

To begin, fix $\eps$ sufficiently small and note that 
$w_\eps$ is an equilibrium solution of the evolution equation
\[
u_t+\partial_x\left(\mathcal{M}_\beta u - c_\eps u + u^2\right)=0.
\]
Linearizing about the equilibrium solution $w_\eps$ leads to the following linear evolution equation:
\begin{equation}\label{lin}
v_t+\partial_x\left(\mathcal{M}_\beta v - c_\eps v + 2w_\eps v\right)=0.
\end{equation}
Here we require $v(\cdot,t)\in L^2(\RM)$ for all $t\geq 0$.  The solution $w_\eps$ is said to be \emph{linearly stable} provided solutions
of \eqref{lin} that begin small remain small for all time.  

A first step in the study of linear stability is often to study the spectrum of the associated linear operator.
Here, this corresponds to studying the spectrum of
\[
L_\eps:=\partial_x\left(\mathcal{M}_\beta  - c_\eps + 2w_\eps \right),
\]
considered as a closed, densely defined linear operator on $L^2(\RM)$ with domain $H^{3/2}(\RM)$.  The wave $w_\eps$ is said to be \emph{spectrally stable} provided the spectrum $\sigma(L_\eps)$
does not intersect the open right half plane, i.e. provided
\[
\sigma(L_\eps)\cap\left\{z\in\mathbb{C}:\Re(z)>0\right\}=\emptyset.
\]
Note that since the coefficients of $L_\eps$ are real-valued, the set $\sigma(L_\eps)$ is symmetric about the real axis.
Furthermore, since $L_\eps$ is the composition of a skew-adjoint and a self-adjoint operator, the spectrum of $L_\eps$ is invariant with respect to reflection through the origin.  Together, this implies the 
following:
\[
\lambda\in\sigma(L_\eps)\quad\Rightarrow\quad \pm\lambda,~\pm\bar\lambda\in\sigma(L_\eps).
\]
It follows that the pattern $w_\eps$ is spectrally stable if and only if $\sigma(L_\eps)\subset\RM i$.

To study the spectrum of $L$, note that from Theorem \ref{main theorem} the wave $w_\eps$ can be decomposed into an exponentially localized ``core" $\Psi_\eps$ and
the oscillatory ``ripple" $P_\eps$:
\[
w_\eps(x)=\Psi_\eps(x)+P_\eps(x).
\]
%It follows that
%\[
%L_\eps=\partial_x\left(\mathcal{M}_\beta-c+2P_\eps\right) + 2\partial_x\left(\Psi_\eps~\cdot\right).
%\]
This decomposition motivates our main observation regarding the spectral analysis of $L_\ep$, which is the content of the following:

\begin{lemma}
When considered as operators on $L^2(\RM)$ with domains $H^{3/2}(\RM)$, the operator $L_\eps$ is a relatively compact perturbation of the asymptotic
operator
\[
\widetilde{L}_\eps:=\partial_x\left(\mathcal{M}_\beta-c+2P_\eps\right).
\]
\end{lemma}

\begin{proof}
Observe that the difference 
\[
\mathcal{D}_\eps:=L_\eps-\widetilde{L}_\eps=2\partial_x\left(\Psi_\eps~\cdot\right)
\]
defines a closed operator  on $L^2(\RM)$ with domain $H^{1}(\RM)$.  Using that $\Psi_\eps$ and its derivatives decay exponentially fast at spatial infinity, the continuity
of $\mathcal{D}_\eps$ as a map from $H^{3/2}(\RM)$ to $H^{1/2}(\RM)$, and the compactness of the embedding $H^{1/2}(I)$ into $L^2(I)$ for any compact interval $I\subset\RM$ 
one can show that if $\{f_n\}_{n=1}^\infty$ is a bounded sequence in $H^{3/2}(\RM)$, then the sequence $\{\mathcal{D}_\eps(f_n)\}_{n=1}^\infty\subset L^2(\RM)$ 
has a convergent subsequence.    The claim now follows.
\end{proof}

Since the essential spectrum is stable with respect to relatively compact perturbations,  it follows that the essential spectrum of $L_\eps$ acting on $L^2(\RM)$ agrees with the essential spectrum
of the periodic-coefficient linear operator $\widetilde{L}_\eps$ acting on $L^2(\RM)$.  Consequently, we can conclude the generalized solitary wave $w_\eps$ is spectrally
unstable if its small amplitude, periodic oscillations are spectrally unstable.  We now study the spectral stability of these oscillations.

As the operator $\widetilde{L}_\eps$ has periodic coefficients its spectrum can be studied via Floquet-Bloch theory, from which it can be easily shown that
non-trivial solutions of $\widetilde{L}_{\eps}v=\lambda v$
can not be integrable over $\RM$:  at best, they can be bounded over $\RM$, and hence the spectrum of $\widetilde{L}_\eps$ over $L^2(\RM)$ is
purely essential.  In particular, it can be shown that $\sigma(\widetilde{L}_{\eps})$ consists of a countable number of continuous curves in $\mathbb{C}$: see \cite{KP_book} for
details.  Using Floquet-Bloch theory, it can be shown that in a sufficiently small neighborhood of the origin $\lambda=0$, the set $\sigma(\widetilde{L}_{\eps})$ consists
of three curves, all of which pass through the origin: see \cite{HJ15b}.  If all three of these curves are confined to the imaginary axis, we say that the background
periodic wave $P_\eps$ is \emph{modulationally stable}, while it is  \emph{modulationally unstable} otherwise.

The modulational stability of the small amplitude periodic traveling wave solutions $P_\eps$ was recently studied in \cite{HJ15b}.  There, the authors use 
rigorous spectral perturbation theory to establish the following result.

\begin{theorem}[Modulational Stability Index]\label{t:mat_vera}
Fix $\beta\in(0,1/3)$ and $\eps>0$ sufficiently small.  Then the $2\pi/{k_\eps^*}$- periodic traveling wave $P_\eps$ is modulationally unstable if $\Delta_{\rm MI}(k_\eps^*)<0$, where
\[
\Delta_{\rm MI}(z):=\frac{(zm_{\beta}(z))''\left((zm_\beta(z))'-m_\beta(0)\right)}{m_\beta(z)-m_\beta(2z)}~~\Delta_{\rm BF}(z),
\]
and $\Delta_{\rm BF}(z):=2(m_\beta(z)-m_\beta(2z))+((zm_\beta(z))'-m(0))$.  Furthermore, $P_\eps$ is modulationally stable if $\Delta_{MI}(k_\eps^*)>0$.
\end{theorem}

It follows from Theorem \ref{t:mat_vera} that there are four mechanisms which can cause a change in the sign of the modulational instability index $\Delta_{\rm MI}$, hence
signaling a change in the modulational stability of the wave $P_\eps$:
\begin{itemize}
\item[(1)] the group velocity $c_g(k):=(km_\beta(k))'$ attains an extremum at some wave number $k$, i.e. $c_g'(k)=0$;
\item[(2)] the group velocity coincides with the phase velocity $c_p(k):=m_\beta(k)$ of the limiting long wave at $k=0$, resulting in a resonance
between long and short waves, i.e. $(km_\beta(k))'=m(0)$;
\item[(3)] the phase velocities of the fundamental mode and the second harmonic coincide, i.e. $m(k)=m(2k)$;
\item[(4)] $\Delta_{BF}(k)=0$.
\end{itemize}
It is interesting to note that possibilities (1)-(3) are purely linear, not depending on any nonlinear effects.  
Note since the waves $P_\eps$ are necessarily supercritical, the third possibility above can never occur.	Furthermore, the formula
for $\Delta_{\rm MI}$ is completely explicit in terms of the phase speed $m_\beta$, and hence can be analyzed numerically.

\begin{figure}[t]
\begin{center}
\includegraphics[scale=0.6]{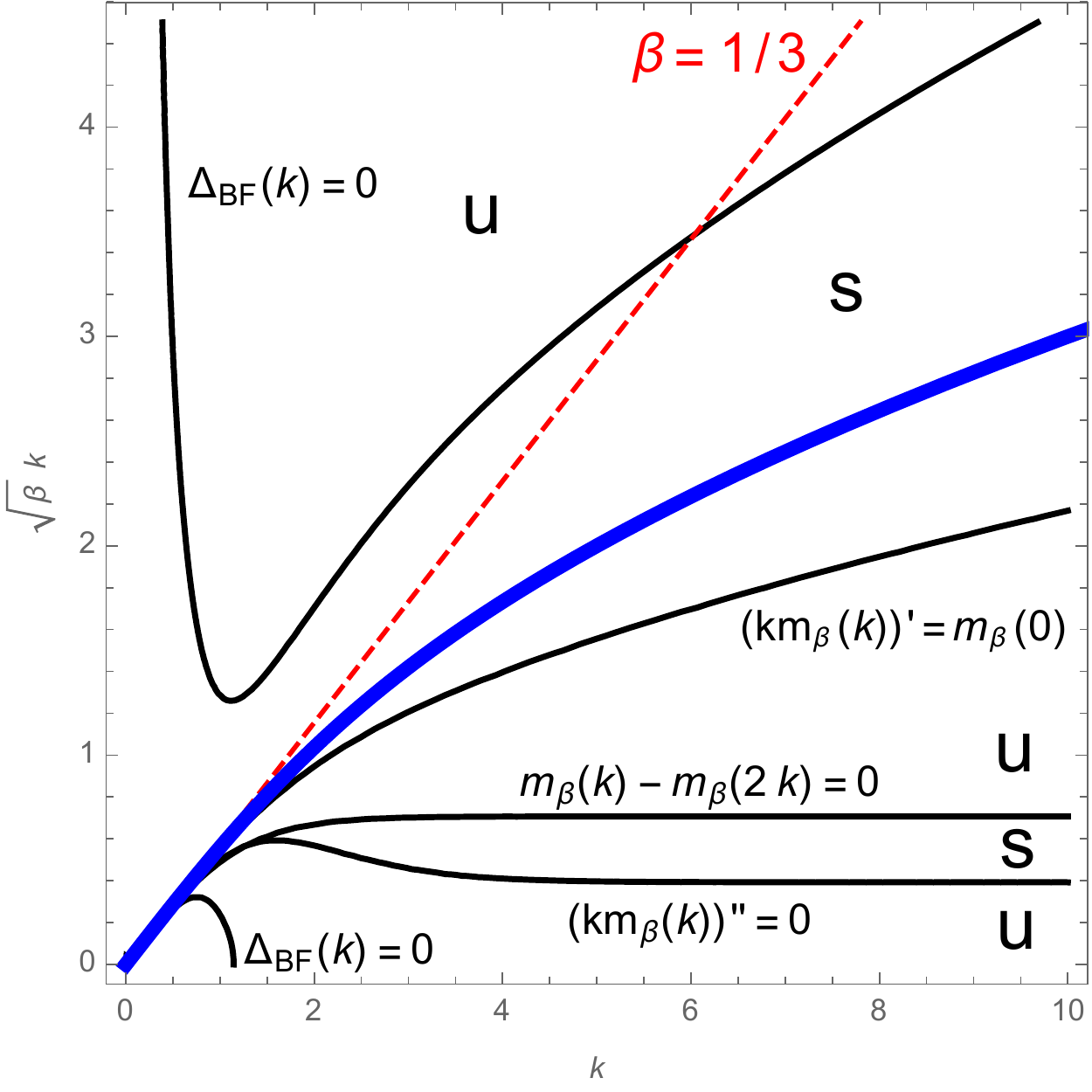}
\caption{The modulational stability diagram, written in coordinates $k$ versus $k\sqrt{\beta}$-plane, for the small amplitude, periodic traveling waves
of \eqref{dimW}.  To interpret, for fixed $\beta>0$ one must envision a line through the origin with slope $\sqrt{\beta}$.  Here, ``$S$" and ``$U$" denote regions
of modulational stability and instability, respectively.  The dark blue curve represents the function $\beta\to k_{\beta,c_0}$, corresponding to the frequencies
of the oscillatory (linear) wave $P_0$.}\label{f:MI}
\end{center}
\end{figure}

According to Theorem \ref{main theorem}, the frequency of the oscillatory wave $P_\eps$ is given by $k_\eps^*\approx k_{\beta,c_\eps}+\mathcal{O}(\eps)$.
Consequently, at first pass we can gain insight into the modulational stability of $P_\eps$ by calculating $\Delta_{\rm MI}(k_{\beta,c_0})$.  Performing such a numerical
calculation indicates that $\beta\mapsto\Delta_{\rm MI}(k_{\beta,c_0})>0$ for all $\tau\in(0,1/3)$: see Figure \ref{f:MI}.

From the above numerical observations, it seems likely that the oscillatory, asymptotic end states $P_\eps$ of the generalized solitary waves $w_\eps$
are modulationally stable for all $\tau\in(0,1/3)$ and $\eps>0$ sufficiently small.  We note that this serves as evidence that the waves
$w_\eps$ could be spectrally stable to localized perturbations in $L^2(\RM)$.  Of course, one must additionally study the spectral stability of the
oscillatory tail $P_\eps$ away from the origin in order to have a complete picture of the essential spectrum of $w_\eps$.  Provided that this analysis
indicates stability, it then remains to understand the effect of the localized core $\Psi_\eps$ on the spectral properties of the linearized operator.  We consider
these as \emph{very} interesting questions that will hopefully be studied elsewhere.

\appendix

\section{Proofs of Technical Estimates}
In this Appendix, we prove a number of technical lemmas used throughout the paper.  To prove Lemma \ref{J0 lemma}, we need the following general result:

\begin{lemma}\label{trunc lemma}
Suppose that $h(Z)$ is a complex valued function with the following properties:
\begin{enumerate}[(i)]
\item $h(Z)$ is analytic on the closed strip $\overline{\Sigma}_q = \left\{ |\Im Z | \le q\right\} \subset \C$ where $q > 0$;
\item there exists $0 \le n \in \mathbb{Z}$ and $c_*>0$ such that $Z \in \overline{\Sigma}_q$ imply that $|h(Z)| \le c_* |z|^n$.
\end{enumerate}
Then the Fourier multiplier operator $\H$ with symbol $h$ is a bounded map from $H^{r+n}_q$ into $H^r_q$ and we have
\be\label{truncate}
\| \H f\|_{r,q} \le C c_*\| f^{(n)}\|_{r,q}.
\ee
The constant $C>0$ does not depend on $h$, $r$, $q$, $n$, or $f$.
\end{lemma}
\begin{proof}
One can show that 
$$
| f |_{r,q} := \sqrt{ \int_\R (1+|K|^2)^r \left( |\hat{f}(K+iq)|^2 + |\hat{f}(K-iq)|^2\right) dK}
$$
is equivalent to $\| f \|_{r,q}$ by means of the Paley-Wiener theorem (see, for instance, \cite{Beale80}). The constants in the equivalence
can be taken independent of both $r$ and $q$.
Then we have, using the equivalence and the estimate in part (ii):
\bes\begin{split}
\|\H f \|_{r,q} \le &C| \H f |_{r,q}\\ =&C
\sqrt{ \int_\R (1+|K|^2)^r \left( |h(K+iq) \hat{f}(K+iq)|^2 + |h(K+iq)\hat{f}(K-iq)|^2\right) dK}\\
\le & Cc_*\sqrt{ \int_\R (1+|K|^2)^r|K|^{2n} \left( |\hat{f}(K+iq)|^2 + |\hat{f}(K-iq)|^2\right) dK}\\
\le & Cc_*| f^{(n)}|_{r,q}\\
\le & C c_* \| f^{(n)}\|_{r,q}.
\end{split}
\ees
\end{proof}

Now we can prove Lemma \ref{J0 lemma}:

\begin{proof} (Lemma \ref{J0 lemma})
%
%Using the definition of $\J_0$ and that fact that $\sigma_\beta(X)$ solves \eqref{KDVTWE}
%shows that
%$$
%\J_0 = \gamma_\beta(\sigma_\beta'' -\sigma_\beta) - \ep^{-2} \left(\M_\beta^\ep -1 - \gamma_\beta \ep^2 \right)\sigma_\beta= -\ep^{-2} ( \M_\beta^\ep-1-\gamma_\beta\ep^2 \partial_X^2)\sigma_\beta.
%$$
A straightforward Taylor's theorem argument shows that there exist $Q>0$ and $c_*>0$ such that
$$
|\Im_\beta(z)|\le Q \implies  \left| m_\beta(z) - 1 + \gamma_\beta z^2  \right| \le c_*|z|^4.
$$
This implies that
$$
|\Im_\beta(Z)|\le Q/\ep \implies  \ep^{-2}\left| m_\beta(\ep Z ) - 1 + \gamma_\beta \ep^2 Z^2  \right| \le c_*\ep^2 |Z|^4.
$$
Note that $\ep^{-2} ( \M_\beta^\ep-1-\gamma_\beta\ep^2 \partial_X^2)$ is a Fourier multiplier with symbol $\ep^{-2}(m_\beta(\ep Z) - 1 + \gamma_\beta \ep^2 Z^2)$ 
and so Lemma \ref{trunc lemma} and \eqref{J0 revisited} gives us
$$
\| \J_0\|_{r,q} \le c_* \ep^{2} \| \sigma_\beta'' \|_{r,q} 
$$
when $q \le Q/\ep$. So, for $\ep$ small enough, we find  for all $q\in[0,q_0]$ and $r\geq 0$ that
$$
\| \J_0\|_{r,q} \le C_r\ep^2.
$$

\end{proof}

Now we will prove Lemmas \ref{suff} and \ref{desing lemma}. Both invoke the following theorem of Beale, from \cite{Beal91}.
\begin{theorem} \label{beale1}
Suppose that $h(Z)$ is a complex valued function which has the following properties:
\begin{enumerate}[(i)]\item
$h(Z)$ is meromorphic on 
the closed strip $\overline{\Sigma}_q = \left\{ |\Im Z | \le q\right\} \subset \C$ where $q > 0$;
\item
there exists $m \ge 0$ and $c_*,\zeta_*>0$ such that $|Z|>\zeta_*$ and $Z \in \overline{\Sigma}_q$ imply
$|h(Z)|\le c_*/|\Re Z|^{m}$;
\item the set of singularities of $h(z)$ in $\overline{\Sigma}_q$ (which we denote $P_h$) is finite
and, moreover, is contained in the interior $\Sigma_q$;
\item all singularities of $h(Z)$ in $\overline{\Sigma}_q$ are simple poles.
\end{enumerate}
Let 
$$
U^r_{h,q}:= \left\{ f \in H^r_q : Z \in P_h  \implies  \hat{f}(Z) = 0\right\}.
$$
Then the Fourier multiplier operator $\H$ with symbol $h$ 
is a bounded injective map from $U^r_{h,q}$ into $H^{r+m}_q$. Additionally, for all $m'\in[0,m]$, 
we have the estimates:
\be\label{mu inv est}
\| \H f \|_{r+m',q} \le \sup_{K \in \R} \left \vert 
{(1+K^2)^{m'/2}  h(K\pm iq)}
\right \vert
 \| f\|_{r,q}.
\ee
\end{theorem}

The main estimates we need to apply the above to prove
 Lemmas \ref{suff} and \ref{desing lemma} are contained in the following:
\begin{lemma}\label{pole analysis}
There exists $\ep_0>0$ and $q_0>0$ such for all $q \in (0,q_0]$ there exists $C_q>0$ for which $\ep \in (0,\ep_0]$ implies
\be\label{smoothing estimate}
\sup_{K \in \R} \left \vert 
{(1+K^2)^{1/4}  l^{-1}_\ep(K + iq)}
\right \vert \le C_q 
\ee
and
\be\label{desingularization estimate}
\sup_{K \in \R} \left \vert
 l^{-1}_\ep(K + iq)+\gamma_\beta^{-1}(1+(K+iq)^2)^{-1}
\right \vert \le C_q \ep.
\ee
\end{lemma}
\begin{proof}
Define $k_\ep := \ep K_\ep  = k_{1,1+\gamma_\beta \ep^2}$.
By definition $m_\beta(k_\ep) = 1+\gamma_\beta \ep^2$.
Also put
$$
\eta_1(z):=m_\beta(z)-1 +\gamma_\beta z^2 \mand \eta_2(z) := m_\beta(z) - m_\beta(k_\ep)-m'_\beta(k_\ep)(z-k_\ep).
$$
These functions are just the remainders in Taylor expansions of $m_\beta(z)$ about $0$ and $k_\ep$ respectively.
We claim that there exists $b,\ep_0,C_1,C_2,C_3> 0$, $\rho \in (0,1)$ and $0<k_1<k_2$
 such that the following all hold when $\ep \in(0,\ep_0]$:
\be
\label{near zero}|\Re(z)| \le k_1\mand |\Im_\beta(z)| \le b \implies |\eta_1(z)| \le C_1 |z|^4 \le {\gamma_\beta \rho \over 2} |z|^2.
\ee
\be\label{near pole}
k_1\le|\Re(z)| \le k_2\mand |\Im_\beta(z)| \le b \implies |\eta_2(z)| \le C_2 |z-k_\ep|^2.
\ee
\be\label{near infinity}
|\Re(z)| \ge k_2\mand |\Im_\beta(z)| \le b \implies |m_\beta(z)-1 -\gamma_\beta\ep^2| \ge C_3 (1+|z|^2)^{1/4}.
\ee
\be\label{dumb equiv}
|\Im_\beta(Z)| \le {1 \over 2} \implies |1+Z^2| \ge \rho(1+|Z|^2).
\ee
Each of these can proved with Taylor's theorem and other differential calculus methods. So we omit the details.

{\it Estimates near $Z =0$:}
We begin by estimating
$$
B(Z):=\left \vert
l_\ep^{-1} (Z) + \gamma_\beta^{-1}(1+Z^2)^{-1}
\right \vert
$$
when $|\Re(Z)| \le k_1/\ep$ and  $|\Im_\beta(Z)| \le b/\ep$.
It is clear that
$$
B(Z)  = \left \vert
l_\ep(Z) + \gamma_\beta (1+Z^2) \over 
l_\ep(Z)  \gamma_\beta (1+Z^2) 
\right \vert.
$$
Recalling the definition of $l_\ep$ tells us
$$
B(Z)  = \left \vert
m_\beta(\ep Z) -1 - \gamma_\beta \ep^2+ \ep^2 \gamma_\beta (1+Z^2) \over 
(m_\beta(\ep Z) -1 - \gamma_\beta \ep^2)\gamma_\beta (1+Z^2) 
\right \vert.
$$
Then we use the definition of $\eta_1(z)$ from above to obtain
$$
B(Z)  = \left \vert
\eta_1(\ep Z)\over 
(\eta_1(\ep Z) - \gamma_\beta \ep^2 (1+Z^2))\gamma_\beta (1+Z^2) 
\right \vert.
$$

The reverse triangle inequality implies
$$
|\eta_1(\ep Z) - \gamma_\beta \ep^2 (1+Z^2)| \ge \gamma_\beta \ep^2 |1+Z^2| - |\eta_1(\ep Z)|.
$$
Then  \eqref{near zero} and \eqref{dumb equiv}  imply 
$$
|\eta_1(\ep Z) - \gamma_\beta \ep^2 (1+Z^2)| \ge {1 \over 2} \gamma_\beta \ep^2 \rho \left(1+ |Z|^2\right).
$$
Note that to use \eqref{dumb equiv} we need $\ep$ such that $|\Im Z| \le 1/2$. Since all the other estimates
hold for $|\Im Z| \le b/\ep$, this is no major constraint.

With this, \eqref{near zero} and \eqref{dumb equiv} we see
\bes%\label{first estimate}
B(Z)  \le { 2 C_1 \ep^2 |Z|^4 \over \gamma_\beta^2  \rho^2 \left(1+ |Z|^2\right)^2 } \le C\ep^2
\ees
when $|\Re(Z)| \le k_1/\ep$ and $|\Im Z| \le 1/2$. This implies that
\be\label{near zero estimate 1}
|q|\le1/2 \implies \sup_{|K|\le k_1/\ep} \left \vert
 l^{-1}_\ep(K + iq)+\gamma_\beta^{-1}(1+(K+iq)^2)^{-1}
\right \vert \le C \ep^2.
\ee

Also, it should be evident that
\bes%\label{interior estimate 1}
(1+\Re(Z)^2)^{1/4} \le C \ep^{-1/2} \mand 
(1+\Re(Z)^2)^{1/4} |1+Z^2|^{-1} \le C
\ees
when
$|\Re(Z)| \le k_1/\ep$ and $|\Im Z| \le 1/2$.
These, the triangle inequality and some naive estimates allow us to conclude that
\be\label{near zero estimate 2}
|q|\le 1/2 \implies \sup_{|K| \le k_1/\ep } \left \vert 
{(1+K^2)^{1/4}  l^{-1}_\ep(K + iq)}
\right \vert \le C. 
\ee

{\it Estimates near $|Z| =K_\ep$:}
The defintions of $l_\ep$ and $\eta_2$ imply
$$
l_\ep^{-1}(Z)=  {\ep^2 \over m'_\beta(k_\ep)(\ep (Z-K_\ep)) + \eta_2(\ep Z)}.
$$
Some algebra takes us to
$$
l_\ep^{-1}(Z)= {\ep \over m'_\beta(k_\ep)(Z-K_\ep)\left( 1 + {\eta_2(\ep Z) \over m'_\beta(k_\ep)(\ep (Z-K_\ep))}\right)}.
$$
If we take $\ep$ sufficiently small then \eqref{near pole} implies, by the geometric series, that
$$
k_1/\ep\le|\Re(Z)| \le k_2/\ep \mand |\Im_\beta(Z)| \le b/\ep \implies 
\left \vert {1 \over  1 + {\eta_2(\ep Z) \over m'_\beta(k_\ep)(\ep (Z-K_\ep))} }\right \vert \le C.
$$
It is clear also that 
$$
k_1/\ep\le|\Re(Z)| \le k_2/\ep  \implies 
\left \vert {1 \over m'_\beta(k_\ep)(K+iq-K_\ep)}\right \vert \le Cq^{-1}.
$$
In this way we find that
\bes%\label{middle estimate 1}
|q| \le 1/2 \implies \sup_{k_1/\ep \le K \le k_2 \ep} |l_\ep^{-1}(K+iq)| \le C\ep q^{-1} \le C_q \ep.
\ees

Also
\bes%\label{interior estimate 1}
(1+\Re(Z)^2)^{1/4} \le C \ep^{-1/2} \mand 
 |1+Z^2|^{-1} \le C \ep^2
\ees
when
$k_1/\ep \le |\Re(Z)| \le k_2/\ep$ and $|\Im Z| \le 1/2$.
These, the triangle inequality and some naive estimates allow us to conclude that
\be\label{middle estimate 2}
|q|\le 1/2 \implies \sup_{k_1/\ep \le |K| \le k_2/\ep } \left \vert 
{(1+K^2)^{1/4}  l^{-1}_\ep(K + iq)}
\right \vert \le C_q \ep^{1/2}. 
\ee
and
\be\label{middle estimate 1}
|q|\le1/2 \implies \sup_{k_1/\ep \le |K| \le k_2/\ep} \left \vert
 l^{-1}_\ep(K + iq)+\gamma_\beta^{-1}(1+(K+iq)^2)^{-1}
\right \vert \le C_q \ep.
\ee

{\it Estimates for  $|Z|\gg0$:} 
Using the definition of $l_\ep$ and \eqref{near infinity} we have
\be\label{outer estimate 0}
k_2/\ep\le|\Re(Z)|  \mand |\Im_\beta(Z)| \le b/\ep \implies 
|l_\ep^{-1}(Z)|  \le C\ep^2 (1+|\ep Z|^2)^{-1/4}.
\ee
Next, it is easy to see that $(1+\Re(Z)^2)/(1+|\ep Z|^2) \le C \ep^{-2}$ when $|\Im_\beta(Z)|<1/2$.
Thus
\be\label{outer estimate 2}
|q|\le1/2 
\implies 
\sup_{K\ge k_2/\ep}| (1+K^2)^{1/4}l_\ep^{-1}(K+iq)|  \le C \ep^{3/2}.
\ee
It is also clear that, if $|\Im Z| < 1/2$, then 
$
\sup_{K\ge k_2/\ep} |1+Z^2|^{-1} \le C \ep^{2}. 
$
Thus \eqref{outer estimate 0} and the triangle inequality tell us
\be\label{outer estimate 1}
|q|\le1/2 
\implies 
\sup_{K\ge k_2/\ep} |l_\ep^{-1}(K+iq) +\gamma_\beta^{-1}(1+(K+iq)^2)^{-1}| \le C \ep^{2}.
\ee

Putting together  \eqref{near zero estimate 1}, \eqref{middle estimate 1} and \eqref{outer estimate 1} gives \eqref{desingularization estimate}.
Putting together  \eqref{near zero estimate 2}, \eqref{middle estimate 2} and \eqref{outer estimate 2} gives \eqref{smoothing estimate}.
We have proven Lemma \ref{pole analysis}.

%Thus we have, for any $|q| \le 1/2$:
%$$
%\sup_{|K| \le \delta_1/2\ep}\left \vert
%l_\ep^{-1} (K+iq) + \gamma_\beta^{-1}(1+(K+iq)^2)^{-1}
%\right \vert \le C_2 \ep^2.
%$$

\end{proof}
\begin{proof} (of Lemma \ref{suff}.)
The estimate \eqref{smoothing estimate} for 
$l_\ep^{-1}(Z)$ permits us to use Theorem \ref{beale1} and this proves Lemma \ref{suff} with no additional complications.
\end{proof}

\begin{proof} (of Lemma \ref{desing lemma}.)
The estimate \eqref{desingularization estimate}, together with the definition of $\P_\ep$, \eqref{pi prop} and the estimate \eqref{Pep estimate}
demonstrate, by way of Theorem \ref{beale1}, that
$$
\| \left(\L_\ep^{-1} +\gamma_\beta^{-1}(1-\partial_X^2)^{-1}\right) \P_\ep\|_{B(E^r_q,E^r_q)} \le C_q \ep.
$$
Thus the triangle inequality tells us that
$$
\|\G_\ep\| = \ep^{-1} \| \L_\ep^{-1}\P_\ep  +\gamma_\beta^{-1}(1-\partial_X^2)^{-1}\|_{B(E^r_q,E^r_q)} \le C_q + C \ep^{-1} \| (1-\partial_X^2)^{-1} \left(\P_\ep - 1 \right)\|_{B(E^r_q,E^r_q)}.
$$
Then we compute, using the definition of $\P_\ep$:
\bes\begin{split}
  \| (1-\partial_X^2)^{-1} \left(\P_\ep - 1 \right)\|_{B(E^r_q,E^r_q)} 
 =& \sup_{ F \in E^r_q, \|F\|_{r,q} =1} \| (1-\partial_X^2)^{-1} \left(\P_\ep - 1 \right) F\|_{r,q}\\
 =& \sup_{ F \in E^r_q, \|F\|_{r,q} =1} \| \left(\P_\ep - 1 \right) F\|_{r-2,q}\\
 =& \sup_{ F \in E^r_q, \|F\|_{r,q} =1} 2\chi_\ep^{-1} |\hat{F}(K_\ep)| \| \sigma_\beta \Phi_\ep^0 \|_{r-2,q}.
\end{split}\ees
Recalling \eqref{RL}, one has $ |\hat{F}(K_\ep)| \le C_{q} \ep^r \| F\|_{r,q}$. And, since
$\sigma_\beta(X) = s_1 \sech^2(s_2 X)$ and $\Phi_\ep^0 (X) = \cos(K_\ep X)$ (with $K_\ep = \O(1/\ep)$) we see that
$ \| \sigma_\beta \Phi_\ep^0 \|_{r-2,q} \le C_r \ep^{2-r}$.
$$
  \| (1-\partial_X^2)^{-1} \left(\P_\ep - 1 \right)\|_{B(E^r_q,E^r_q)}  \le C_r \ep^2 .
$$
And so we have
$
\|\G_\ep\|  \le C_q + C_r \ep \le C_{r,q}.
$
\end{proof}

Now we address Lemmas \ref{J2 lemma} and \ref{J3 lemma}. These are modeled on the proofs for the estimates found in Appendix E.4 of \cite{Faver-Wright}.
\begin{proof}(Lemma \ref{J2 lemma}).
We only address the second estimate since it implies the first.
First:
$$
\| \J_2 - \tilde{\J}_2\|_{r,q} = 2 \|\sigma_\beta \left(a (\Phi^a_\ep-\Phi^0_\ep) - \tilde{a} (\Phi^{\tilde{a}}_\ep-\Phi^0_\ep)\right) \|_{r,q}
$$
Then we use triangle inequality:
\be\label{here}
\| \J_2 - \tilde{\J}_2\|_{r,q} = 2|a| \|\sigma_\beta  (\Phi^a_\ep -\Phi^{\tilde{a}}_\ep)\|_{r,q}+ 2|a -\tilde{a}|\| \sigma_\beta (\Phi^{\tilde{a}}_\ep-\Phi^0_\ep) \|_{r,q}.
\ee

Next we recall the definition of $\Phi_\ep^a$ in \eqref{this is PHI} to get:
\be\label{annoying1}\begin{split}
\left \vert 
\Phi_\ep^a (X)- \Phi_\ep^{\tilde{a}}(X)
\right \vert &= \left \vert\phi_\ep^a (K_\ep^a X)- \phi_\ep^{\tilde{a}}(K_\ep^{\tilde{a}}X) \right \vert\\
& \le \left \vert\phi_\ep^a (K_\ep^a X) - \phi_\ep^{{a}}(K_\ep^{\tilde{a}}X) \right \vert + \left \vert\phi_\ep^a (K_\ep^{\tilde{a}} X) - \phi_\ep^{\tilde{a}}(K_\ep^{\tilde{a}}X) \right \vert.
\end{split}
\ee
The second term above can be estimated using \eqref{philip} in Theorem \ref{MV} to see
$$
\left \vert\phi_\ep^a (K_\ep^{\tilde{a}} X) - \phi_\ep^{\tilde{a}}(K_\ep^{\tilde{a}}X) \right \vert \le C|a-\tilde{a}|,
$$
independent of $X$.
We can use Taylor's theorem to control the first term in \eqref{annoying1}:
$$
\left \vert\phi_\ep^a (K_\ep^a X) - \phi_\ep^{{a}}(K_\ep^{\tilde{a}}X) \right \vert \le |K^a_\ep - K^{\tilde{a}}_\ep||X|
\|\partial_y \phi_\ep^a\|_{L^\infty}.
$$
Then we deploy \eqref{Klip} and \eqref{philip}
$$
\left \vert\phi_\ep^a (K_\ep^a X) - \phi_\ep^{{a}}(K_\ep^{\tilde{a}}X) \right \vert  \le C|a-\tilde{a}| |X| 
$$
for all $X$.
Thus we have
$$
\left \vert 
\Phi_\ep^a (X)- \Phi_\ep^{\tilde{a}}(X)
\right \vert \le C|a-\tilde{a}|(1+|X|).
$$
for any $X$.
The same sort of reasoning in a longer and more annoying argument can be used to show that
\be\label{annoying}
\left \vert 
\partial_X^r(\Phi_\ep^a (X)- \Phi_\ep^{\tilde{a}}(X))
\right \vert \le C_r \ep^{-r} |a-\tilde{a}|(1+|X|)
\ee
holds for all $X \in \R$.

Using \eqref{annoying} in \eqref{here}
gives
\be\begin{split}
\| \J_2 - \tilde{\J}_2\|_{r,q} \le C_r \ep^{-r} (|a|+|\tilde{a}|)|a-\tilde{a}| \|\sigma_\beta  (1+|\cdot|)\|_{r,q}
\end{split}
\ee
And since $\sigma_\beta$ is just a scaled $\sech^2$ function we have
\be\begin{split}
\| \J_2 - \tilde{\J}_2\|_{r,q} \le C_r \ep^{-r}(|a|+|\tilde{a}|)||a-\tilde{a}| .
\end{split}
\ee
\end{proof}

\begin{proof} (Lemma \ref{J3 lemma}) 
The second estimate is more complicated than the first, so we only prove it. Note, however, that the second does not imply the first.
First:
$$
\| \J_3 - \tilde{\J}_3\|_{r,0} = 2 \| a R {\Phi}_\ep^a - \tilde{a} \tilde{R} {\Phi}_\ep^{\tilde{a}}   \|_{r,0}
$$
The triangle inquality gives us
$$
\| \J_3 - \tilde{\J}_3\|_{r,0} \le  2|a| \|  R ({\Phi}_\ep^a - {\Phi}_\ep^{\tilde{a}})\|_{r,0} + 2|a-\tilde{a}|\|R {\Phi}_\ep^{\tilde{a}} \|_{r,0} + |\tilde{a}| \| (R-\tilde{R}) {\Phi}_\ep^{\tilde{a}}   \|_{r,0}.
$$
The last term on the right hand side of this is easily estimated by $C_r \ep^{-r}|\tilde{a}|\| R- \tilde{R}\|_{r,0}$.

The first two terms on the right hand side above can be handled almost identically to how we dealt with the terms on the right hand side in \eqref{here}, but with $R$ replacing
$\sigma_\beta$. We find that
$$
2|a| \|  R ({\Phi}_\ep^a - {\Phi}_\ep^{\tilde{a}})\|_{r,0} + 2|a-\tilde{a}|\|R {\Phi}_\ep^{\tilde{a}} \|_{r,0} \le C_r\ep^{-r} |a-\tilde{a}| \| R(1+|\cdot|)\|_{r,0}.
$$
Since we are assuming $R \in E^r_q$ with $q>0$, we have
$
 \| R(1+|\cdot|)\|_{r,0} \le C_{r,q} \| R\|_{r,q}.
$
Thus all together we find
$$
\| \J_3 - \tilde{\J}_3\|_{r,0} \le  C_{r,q} \ep^{-r} \left( (\|R\|_{r,q}+\|\tilde{R}\|_{r,q})|a-\tilde{a}| +  (|a|+|\tilde{a}|)\|R-\tilde{R}\|_{r,0}\right).
$$

\end{proof}

\bibliographystyle{plain}
\bibliography{CapWhitham}

\end{document}